\documentclass[a4paper,10pt]{amsart}

\usepackage[english]{babel}

\usepackage{verbatim}

\usepackage{amssymb} %for \Bbbk

\usepackage{tikz, tikz-3dplot} % for graphs

\usepackage{lmodern} % Paket zur Verwendung einer verbesserten Schriftart

\usepackage{fancyvrb} % Paket zur besseren Implementierung von Code in den Fließtext

\usepackage{url} % Paket zur Einbindung von URL-Links

\usepackage{hyperref} % Paket zur Verwendung von Links im Dokument

\newtheorem{thm}{Theorem}[section]
\newtheorem{trm}[thm]{Theorem}

\newtheorem{prop}[thm]{Proposition}

\newtheorem{lma}[thm]{Lemma}

\theoremstyle{definition}

\newtheorem{defi}[thm]{Definition}

\newtheorem{ex}[thm]{Example}

\theoremstyle{remark}
\newtheorem{rem}[thm]{Remark}

\makeatletter
\let\c@equation\c@thm
\makeatother
\numberwithin{equation}{section}

\newcommand{\C}{\mathbb{C}}

\newcommand{\Z}{\mathbb{Z}}
\newcommand{\N}{\mathbb{N}}
\newcommand{\Kbb}{\Bbbk}
\newcommand{\f}{\textbf}
\newcommand{\mr}{\mathrm}
\newcommand{\mb}{\mathbf}
\newcommand{\bdim}{\mathbf{dim}\,}
\newcommand{\fa}{ \ \mathrm{for}\ \mathrm{all}\ }
\newcommand{\Hom}{\mathrm{Hom}}

\title[Degenerate Affine Flag Varieties and Quiver Grassmannians]{Degenerate Affine Flag Varieties and Quiver Grassmannians}

\author[Alexander P\"{u}tz]{Alexander P\"{u}tz}

\address{Ruhr-University Bochum, Faculty of Mathematics, Universit\"{a}tsstrasse 150, 44780 Bochum (Germany)}

\email{alexander.puetz@rub.de}

\begin{document}
% Abstract ---------------------------------------------------------------------------------

\begin{abstract}
We study finite dimensional approximations to degenerate versions of affine flag varieties using quiver Grassmannians for cyclic quivers. We prove that they admit cellular decompositions parametrized by affine Dellac configurations, and that their irreducible components are normal Cohen-Macaulay varieties with rational singularities.
\end{abstract}

% Titel and Content ---------------------------------------------------------------------------------
\maketitle
%\newpage
%\tableofcontents
%\pagenumbering{roman}

%-------------------------------------------------------------------------------------
\section*{Introduction}
%-------------------------------------------------------------------------------------
\pagenumbering{arabic}
This work is based on the identification of the degenerate type A flag variety with a quiver Grassmannian for the equioriented quiver of type A as shown by G.~Cerulli Irelli, E.~Feigin and M.~Reineke in \cite{CFR2012}. Additionally, it is based on a similar construction, using quiver Grassmannians for the loop quiver, giving finite approximations of the degenerate affine Grassmannian as introduced by E.~Feigin, M.~Finkelberg and M.~Reineke in \cite{FFR2017}. We generalise their constructions to describe finite approximations of degenerate affine flag varieties, using quiver Grassmannians for the equioriented cycle. Linear degenerations of the affine flag variety are defined similarly to the construction for the type A flag variety by G.~Cerulli Irelli, X.~Fang, E.~Feigin, G.~Fourier and M.~Reineke \cite{CFFFR2017}.

Quiver Grassmannians were first used by W.~Crawley-Boevey and A.~Schofield \cite{Crawley1989,Schofield1992}. In some special cases quiver Grassmannians for the equioriented cycle were studied by N.~Haupt \cite{Haupt2011,Haupt2012}. The variety of representations of the cycle was studied by G.~Kempken \cite{Kempken1982}. J.~Sauter studied the quiver flag variety for the equioriented cycle \cite{Sauter2017}. The Ringel-Hall algebra of the cyclic quiver was studied by A.~Hubery \cite{Hubery2010}. Based on the work by G.~Kempken and A.~Hubery we derive statements about the geometry of quiver Grassmannians for the equioriented cycle and generalise a result by N.~Haupt about the parametrisation of their irreducible components.

The main goal of this paper is the study of degenerate affine flag varieties of type $\mathfrak{gl}_n$. Analogous to the classical setting the affine flag variety is defined as the quotient of the affine Kac-Moody group by its standard Iwahori subgroup \cite[Chapter XIII]{Kumar2002}. Based on the identification of the finite approximations with quiver Grassmannians for the equioriented cycle, we examine geometric properties of the degenerations.

%-------------------------------------------------------------------------------------
\subsection*{Main Results}
%-------------------------------------------------------------------------------------
Finite approximations of the degenerate affine flag variety are isomorphic to certain quiver Grassmannians for the cyclic quiver, as shown in Theorem~\ref{trm:finite_approx-in-the-intro}. The approximations have equidimensional irreducible components, which are parametrised by grand Motzkin paths, and the components are normal, Cohen-Macaulay and have rational singularities (Theorem~\ref{trm:geom-prop-aff-flag-in-the-intro}). Moreover, certain quiver Grassmannians for the cycle admit a cellular decomposition, into attracting sets of the fixed points of an action by a one-dimensional torus (Theorem~\ref{trm:cell_decomp-approx-lin-deg-aff-flag}). The cells in the approximations of the degenerate affine flag are parametrised by affine Dellac configurations, as shown in Theorem~\ref{trm:bij-cells-aff-delllac}.
%-------------------------------------------------------------------------------------
\subsection*{Methods and Structure}
%-------------------------------------------------------------------------------------
To utilise the approximations via quiver Grassmannians, it is necessary to understand the quiver Grassmannians for the oriented cycle and their geometric properties. The link between representations of quivers and modules over finite dimensional algebras is the foundation of the realisation of quiver Grassmannians as framed module spaces. This interpretation of quiver Grassmannians allows us to translate properties between the variety of quiver representations and the quiver Grassmannian. Some basic definitions and required constructions are recalled in Section~\ref{sec:quiver-rep}.

In Section~\ref{sec:deg-aff-flag}, we define the affine flag variety of type $\mathfrak{gl}_n$ and its degeneration and state the main theorems of this article. The rest of the article is devoted to prove and generalise these statements.

The equioriented cycle, and the class of quiver Grassmannians which we want to examine, are introduced Section~\ref{sec:quiver-grass-cycle}. Based on word combinatorics we prove a dimension formula for the space of morphisms between nilpotent indecomposable representations of the cycle. This is applied to the parametrisation of irreducible components. Moreover it reveals certain favourable geometric properties of the quiver Grassmannians for the cycle and of the approximations of the degenerate affine flag variety, as claimed in Theorem~\ref{trm:geom-prop-aff-flag-in-the-intro}. The proof of the geometric properties utilises the construction of the quiver Grassmannian as framed moduli space. Hence they lift from the variety of quiver representations, which was studied by G.~Kempken in \cite{Kempken1982}.

In Section~\ref{sec:cell-decomp}, we introduce a $\C^*$-action on the quiver Grassmannians for the equioriented cycle. This action provides us with a combinatorial tool, to compute the Euler characteristic of these quiver Grassmannians, which was introduced by G.~Cerulli Irelli in \cite[Theorem~1]{Cerulli2011}. Moreover, it induces a cellular decomposition of the quiver Grassmannians. For the proof that such a decomposition exists, it is crucial to find the right grading for the basis of the vector spaces belonging to a quiver representation. This grading defines a torus action on the quiver representation and induces a torus action on the quiver Grassmannian. 

For the degenerate affine flag variety, the cells are in bijection with affine Dellac configurations. This bijection is established in Section~\ref{sec:aff-dellac} and is based on the parametrisation of torus fixed points by successor closed subquivers \cite[Proposition~1]{Cerulli2011}. Finally, in Section~\ref{sec:lin-deg-flag}, we generalise the constructions from the previous sections to certain linear degenerations of the affine flag variety and examine their geometry.

The class of quiver Grassmannians studied in this paper has some rather strong restrictions. But deviating from the Dynkin setting, these restrictions are necessary to obtain the results expected from the classical setting. If we drop the restrictions which keep up the analogy, it is hard to say anything about the geometry of the quiver Grassmannians. Only the existence of a cellular decomposition is rather general among the quiver Grassmannians for the cycle. It would be interesting to find out more about the geometry of the linear degenerations, as introduced in Section~\ref{sec:lin-deg-flag}, but this requires new methods, since the identification of the approximations with a framed moduli space fails in this setting. 
%-------------------------------------------------------------------------------------
\section{Representations of Quivers}\label{sec:quiver-rep}
%-------------------------------------------------------------------------------------
In this section we recall some definitions and results about quiver representations which are required to examine properties of approximations of the degenerate affine flag variety. Fix an algebraically closed field $\Kbb$. Let $Q$ be a finite quiver with a finite set of vertices $Q_0$, a finite set of edges $Q_1$ between the vertices and two maps $s,t : Q_1 \to Q_0$ providing an orientation of the edges with source $s_\alpha$ and target $t_\alpha$ for all $\alpha \in Q_1$. A $Q$-representation $R$ is a pair of tuples $R=(V,M)$, with a tuple of $\Kbb$-vector spaces over the vertices $V=(V_i)_{ i \in Q_0}$, and a tuple of linear maps between the vector spaces along the arrows of the underlying quiver $M = (M_\alpha )_{\alpha \in Q_1}$.  

The category of finite dimensional $Q$-representations over the field $\Kbb$ is $\mathrm{rep}_\Kbb(Q)$. A morphism $\psi$ of $Q$-representations $R=(V,M)$ and $S=(U,N)$ is a collection of linear maps $\psi_i : V_i \to U_i$ such that $\psi_j \circ M_\alpha = N_\alpha \circ \psi_j$ for all edges $\alpha : i \to j$. The set of all $Q$-morphisms from $R$ to $S$ is denoted by $\mathrm{Hom}_Q(R,S)$. Its dimension is abbreviated as $[R,S]:= \dim \mathrm{Hom}_Q(R,S)$.
%-------------------------------------------------------------------------------------
\subsection{Quiver Grassmannians}
%-------------------------------------------------------------------------------------
The entries of the dimension vector $\bdim R \in \Z^{Q_0}$ of a quiver representation $R$ are given by the dimension of the vector spaces $V_i$ over the vertices of the quiver. A subrepresentation $S \subseteq R$ is parametrised by a tuple of vector subspaces $U_i \subset V_i$, which is compatible with the maps between the vector spaces of the representation $R$, i.e. for all arrows $\alpha : i \to j$ of $Q$ we have $M_\alpha(U_i) \subseteq U_j$. 
\begin{defi}
The \textbf{quiver Grassmannian} $\mr{Gr}_{\mb{e}}^Q(M)$ is the set of all subrepresentations, with dimension vector $\mb{e} \in \Z^{Q_0}$, of the $Q$-representation $M$.
\end{defi}
In \cite{CFR2012} the stratum of $U \in \mr{Gr}^{\Delta_n}_{\mb{e}}(M)$ is defined as the isomorphism class of $U$ in the quiver Grassmannian. By \cite[Lemma 2.4]{CFR2012}, it is irreducible, locally closed and of dimension
\[ \dim \mathcal{S}_U = [ U,M ]- [ U,U].\] 
Hence the irreducible components of the quiver Grassmannian are given by the closures of the strata which are not contained in the closure of any other stratum.
%-------------------------------------------------------------------------------------
\subsection{Path Algebras and Bounded Representations}
%-------------------------------------------------------------------------------------
A path in a quiver is the concatenation of successive arrows. The path algebra $A:=\Kbb Q$ of the quiver $Q$ is the $\Kbb$-algebra with all paths in $Q$ as basis and multiplication of paths is defined via concatenation \cite[Definition 4.5]{Schiffler2014}. A quiver representation is called indecomposable if it can not be written as direct sum of two non-zero quiver representations. Every finite dimensional quiver representation has a decomposition into indecomposable representations which is unique up to the order of the summands \cite[Theorem~1.11]{Kirillov2016}. $Q$ has finitely many indecomposable representations if and only if the underlying graph of $Q$ is a simply-laced Dynkin diagram \cite{Gabriel1972}. This implies that the path algebra is finite dimensional. 

For arbitrary quivers we have to restrict to representations of bound quivers, i.e. representations of $Q$ such that the maps of the representation satisfy relations from an admissible ideal $\mr{I}$ in the path algebra $\Kbb Q$. Then $A_\mr{I} := \Kbb Q/\mr{I}$ is finite dimensional and in particular there are only finitely many indecomposable representations of the bound quiver $(Q,\mr{I})$ \cite[Theorem~5.4]{Schiffler2014}.
%-------------------------------------------------------------------------------------
\subsection{Variety of Quiver Representations}
%-------------------------------------------------------------------------------------
For a dimension vector $\mb{e} \in \Z^{Q_0}$, define the variety of $Q$-representations as
\[ \mr{R}_\mb{e}(Q) := \bigoplus_{(\alpha: i \to j) \in Q_1} \mathrm{Hom}\big(\Kbb^{e_i},\Kbb^{e_j}\big).\] 
This generalises to representations of bound quivers if we assume that the maps satisfy admissible relations $\mr{I}$ and we write $\mr{R}_\mb{e}(Q,\mr{I})$. On points $M$ in both varieties an element $g$ of the group \[\mr{G}_\mb{e} := \prod_{i \in Q_0}\mr{GL}_{e_i}(\Kbb) \] acts via conjugation, i.e.
\[ g.M := \Big(g_j M_\alpha g_i^{-1} \Big)_{(\alpha :i\to j) \in Q_1}. \]
The dimension of an orbit $\mathcal{O}_M := \mr{G}_\mb{e}.M$ is $\dim \mathcal{O}_M = \dim \mr{G}_\mb{e} - [M,M]$. If we fix a basis for the vector spaces over vertices of $Q$, the isomorphism classes of $Q$-representations are exactly the $\mr{G}_\mb{e}$-orbits. For $M,N \in \mr{R}_\mb{e}(Q,\mr{I})$, we write $N \geq M$ if $\mathcal{O}_N \subseteq \overline{\mathcal{O}_M}$ and say $M$ degenerates to $N$. 
%-------------------------------------------------------------------------------------
\subsection{Simple, Projective and Injective Representations}
%-------------------------------------------------------------------------------------
The simple representation $S_i$ of $Q$ has a one-dimensional $\Kbb$-vector space over the $i$-th vertex. All other vector spaces and the maps along the arrows are zero. The projective representation $P_k$ has a vector space with a basis indexed by paths from $k$ to $i$ over the $i$-th vertex. The maps along the arrows are determined by concatenation of the arrows with the paths labelling the basis. Analogously, the injective representation $I_k$ has a vector space with a basis indexed by paths from $j$ to $k$ over the $j$-th vertex. The maps along the arrows are determined by the factoring of paths through the arrows. For bound quivers, one has to work with the equivalence classes of paths in the bounded path algebra \cite[Definition~5.3]{Schiffler2014}.
%-------------------------------------------------------------------------------------
\subsection{Coefficient Quivers}
%-------------------------------------------------------------------------------------
Let $R=(V,M)$ be a finite dimensional representation of a finite quiver $Q$ and $\mb{d}$ its dimension vector. The coefficient quiver $Q(R)$ has one vertex for each basis element $v^{(i)}_k$ with $k \in [d_i]$ of the vector spaces $V_i$ for $i \in Q_0$. We draw an arrow from a vertex $v^{(i)}_k$ to a vertex $v^{(j)}_\ell$ if $\alpha : i \to j \in Q_1$ and the coefficient of $v^{(j)}_\ell$ in $M_\alpha(v^{(i)}_k)$ is non-zero. A subquiver in $Q(R)$ is called successor closed if for all vertices in the subquiver, their image along the arrows of $Q(R)$ is also contained in the subquiver.
%-------------------------------------------------------------------------------------
\subsection{Framed Moduli Interpretation of Quiver Grassmannians}\label{sec:framed-moduli-int}
%-------------------------------------------------------------------------------------
The \f{extended representation variety} is defined as
\[ \mr{R}_{\mb{e},\mb{d}}(Q,\mr{I}) := \mr{R}_{\mb{e}}(Q,\mr{I}) \times \Hom_\Kbb(\mb{e},\mb{d}) \]
\begin{defi}%[Reineke \cite{Reineke2008}, Definition 2.1]
A point $(M, f )$ of $\mr{R}_{\mb{e},\mb{d}}(Q,\mr{I})$ is called \f{stable} if there is no
non-zero subrepresentation $U$ of $M$, which is contained in $\mr{Ker}f \subseteq M$. The set of all stable points of $\mr{R}_{\mb{e},\mb{d}}(Q,\mr{I})$ is denoted by $\mr{R}^s_{\mb{e},\mb{d}}(Q,\mr{I})$. 
\end{defi}
\begin{trm}
\label{trm:framed-module}
Let $Q$ be a finite connected quiver and $\mr{I}$ an admissible ideal of the path algebra $\Kbb Q$. The indecomposable injective representation of the bound quiver $(Q,\mr{I})$ ending at vertex $j \in Q_0$ is denoted by $I_j$. Then
\[ \mr{Gr}_{\mb{e}}^{Q}(J) \quad \cong \quad \mr{M}^s_{\mb{e},\mb{d}}(Q,\mr{I}), \]
where 
\[ J := \bigoplus_{j \in Q_0} I_j \otimes \Kbb^{d_j}, \]
and $\mr{M}^s_{\mb{e},\mb{d}}(Q,\mr{I})$ is the geometric quotient of $\mr{R}^s_{\mb{e},\mb{d}}(Q,\mr{I})$ by the group $\mr{G}_{\mb{e}}$.
\end{trm}
The existence of this geometric quotient is part of the statement and for its proof it is important that the stability as defined above is a stability in the sense of geometric invariant theory. Here $\mb{d}$ is a tuple with multiplicities of injective bounded quiver representations and not the dimension vector of the quiver representation $J$. This theorem was first proven by M.~Reineke for Dynkin quivers in \cite[Proposition 3.9]{Reineke2008}. S.~Fedotov used the same methods to derive the statement in the generality of modules over finite dimensional algebras in \cite[{Theorem 3.5}]{Fedotov2013}.

The following theorem establishes a bijection between orbits in the variety of quiver representations and strata in the corresponding quiver Grassmannian, which preserves geometric properties. It allows us to lift certain properties of the variety of quiver representation, studied by G.~Kempken, to the quiver Grassmannians, which we use for the finite approximations of the affine flag variety and its degeneration. In this generality, it was already known by K.~Bongartz \cite{Bongartz1997} and in the case of Dynkin quivers it is proven by M.~Reineke in \cite[Theorem~6.4]{Reineke2008}. 

Define $\mr{R}^{(\mb{d})}_{\mb{e}}(Q,\mr{I})$ as the image of the projection
\[ pr: \mr{R}^s_{\mb{e},\mb{d}}(Q,\mr{I}) \to \mr{R}_{\mb{e}}(Q,\mr{I}).\]
\begin{trm}
\label{trm:group-action}
There is a bijection between $\mr{Aut}_{Q}(J)$-stable subvarieties of $\mr{M}^s_{\mb{e},\mb{d}}(Q,\mr{I})$ and $\mr{G}_{\mb{e}}$-stable subvarieties of $\mr{R}^{(\mb{d})}_{\mb{e}}(Q,\mr{I})$ such that inclusions, closures, irreducibility and types of singularities are preserved.
\end{trm}
The proof of this theorem is obtained with the same arguments as used by M.~Reineke, for the Dynkin case in \cite{Reineke2003,Reineke2008}.
%-------------------------------------------------------------------------------------
\section{The Degenerate Affine Flag Variety}\label{sec:deg-aff-flag}
%-------------------------------------------------------------------------------------
In this section we recall the definition of the affine flag variety, define its degeneration and state some geometric properties of the degeneration, which are based on the identification with quiver Grassmannians. Let $\C[[t]]$ be the ring of formal power series over $\C$ and let $\C((t))$ be the field of Laurent series over $\C$. Then 
\[\widehat{\mr{GL}}_n:= \mr{GL}_n(\C) \otimes \C((t))
\]
is the \f{loop group} associated to $\mr{GL}_n(\C)$. Let $B \subset \mr{GL}_n(\C)$ be a Borel subgroup. For example, it can be chosen as the subgroup of upper triangular matrices. Let 
\[ p: \mr{GL}_n(\C) \otimes \C[[t]] \to \mr{GL}_n(\C) \]
be the projection, which evaluates every formal power series at $t=0$. The preimage $I := p^{-1}(B)$ is a subgroup of $\widehat{\mr{GL}}_n$ and is called \f{Iwahori subgroup}.

\begin{defi}
The \f{affine flag variety} of type $\mathfrak{gl}_n$ is defined as
\[ \mathcal{F}l\big(\widehat{\mathfrak{gl}}_n\big) := \widehat{\mr{GL}}_n/I.\]
\end{defi}
\begin{rem}
By \cite[Corollary~13.2.9]{Kumar2002}, this definition coincides with the definition of the affine flag variety, via affine Lie algebras and their associated affine Kac-Moody groups, as mentioned in the introduction. This construction is described in detail in \cite[Chapter~XIII]{Kumar2002} and a brief overview can be found in \cite[Section~1.1]{FFR2017}. For our purposes it is sufficient to have the above version of the definition, since this is the definition used in \cite{KaPe1986}, where the parametrisation as introduced in Proposition~\ref{prop:alt-param-classical} is constructed.
\end{rem}

Before we describe the alternative parametrisation of the affine flag variety, we recall it for the $\mr{SL}_n$-flag variety. Let $B_n$ denote a Borel subgroup of $\mr{SL}_n(\C)$. Then 
\[ \mathcal{F}l\big(\mathfrak{sl}_n\big) := \mr{SL}_n(\C)/B_n \cong \Bigg\{ \big(U_k \big)_{k=1}^{n-1}  \in \prod_{k=1}^{n-1} \mathrm{Gr}_k(\C^n) \ : \  U_1 \subset U_2 \subset \ldots \subset U_{n-1} \Bigg\}. \]
With the parametrisation on the right hand side, it can be realised as quiver Grassmannian \cite[Remark~2.8]{CFR2012}.
%-------------------------------------------------------------------------------------
\subsection{Alternative Parametrisations of the Affine Flag Variety}
%-------------------------------------------------------------------------------------
The first step in the direction of approximations by quiver Grassmannians is an alternative description of the affine flag variety, which is similar to the identification with the set of vector space chains in the classical setting. The affine flag variety is infinite dimensional such that we have to replace the finite dimensional vector space by some infinite dimensional objects. 

There are two approaches to this problem. The more common construction is via lattice chains \cite{BeLa1994,Goertz2001,Goertz2010}. It is possible to define approximations and degenerations of the affine flag variety in this setting, but we want to take a different path, where the analogy to the classical setting is more visible. The second construction is based on Sato Grassmannians \cite{FFR2017,KaPe1986}. 

For $\ell \in \Z$ let $V_\ell$ be the vector space 
\[ V_\ell := \mr{span}(v_\ell,v_{\ell-1},v_{\ell-2},\dots), \] 
which is a subspace of the infinite dimensional $\C$-vector space $V$ with basis vectors $v_i$ for $i \in \Z$. The \textbf{Sato Grassmannian} $\mathrm{SGr}_k$ for $k \in \Z$ is defined as
\[ \mathrm{SGr}_k := \big\{ U \subset V \  :  \ \mr{There} \ \mr{exists} \ \mr{a} \ \ell < k \ \mr{s.t.} \  V_\ell \subset U  
\  \mr{and}  \  \dim U/V_\ell = k-\ell \ \big\}. \]
The vector spaces in the chains for the classical flag variety are elements of the Grassmannians $\mr{Gr}_k(n)$. Analogously, we obtain a description of the affine flag variety as a set of cyclic chains where the vector spaces are elements of the Sato Grassmannians $\mr{SGr}_k$.
\begin{prop}[\cite{FFR2017,KaPe1986}]
\label{prop:alt-param-classical}
The affine flag variety $\mathcal{F}l\big(\widehat{\mathfrak{gl}}_n\big)$ as subset in the product of Sato Grassmannians is parametrised as
\[\mathcal{F}l\big(\widehat{\mathfrak{gl}}_n\big) \cong \Bigg\{ \big(U_k \big)_{k=0}^{n-1}  \in \prod_{k=0}^{n-1} \mathrm{SGr}_k  \ : \  U_0 \subset U_1 \subset \ldots \subset U_{n-1} \subset s_n U_0 \Bigg\}, \]
where $s_n : V \to V$ maps $v_i$ to $v_{i+n}$ for all $i \in \Z$.
\end{prop}
\begin{rem}
Here we use a slightly different parametrisation for the affine flag variety than the version by E.~Feigin, M.~Finkelberg and M.~Reineke in \cite{FFR2017}. It turns out that this parametrisation fits better for the identification of finite approximations with quiver Grassmannians which we have in mind. 
\end{rem}
It is shown by E.~Feigin in \cite[Theorem~0.1]{Feigin2011} that the degenerate $SL_{n}$-flag variety admits a description via vector space chains, where the spaces are related by projections instead of inclusions. In the same way, we define degenerations of the affine flag variety.
\begin{defi}
\label{def:deg-aff-flag}
The \textbf{degenerate affine flag variety} $\mathcal{F}l^a\big(\widehat{\mathfrak{gl}}_n\big)$ is defined as
\[\mathcal{F}l^a\big(\widehat{\mathfrak{gl}}_n\big)  := \Bigg\{ \big(U_k \big)_{k=0}^{n-1} \in \prod_{k=0}^{n-1} \mathrm{SGr}_k \  : \ \mr{pr}_{k+1} U_k \subset U_{k+1}, \   \mr{pr}_{n}U_{n-1} \subset s_n U_0 \Bigg\}, \]
where $\mr{pr}_{i} : V \to V$ is the projection of $v_i$ to zero.
\end{defi}
\begin{rem}\label{rem:isomorphic-def-of-deg}
The different parametrisation of the affine flag also leads to a different definition of the degeneration, but the version as defined above is isomorphic to the degenerate affine flag variety as introduced in \cite[Definition~1.3]{FFR2017}.
\end{rem}
The isomorphism from Remark~\ref{rem:isomorphic-def-of-deg} is described below. Translated to our notation, Definition~1.3 in \cite{FFR2017} reads as
\[\mathcal{F}l^a\big(\widehat{\mathfrak{gl}}_n\big)  = \Bigg\{ \big(U_k \big)_{k=0}^{n} \in \prod_{k=0}^{n} \mathrm{SGr}_k \  : \ \mr{pr}_{k+1} U_k \subset U_{k+1} \ \  \mr{and} \ \ U_{n} = s_n U_0 \Bigg\}. \]
This follows from their identification of basis of $V$ and $\C^n \otimes \C[t,t^{-1}]$ by 
\[ v_{nk+j} = e_j \otimes t^{-k-1} \quad \mr{for} \ j \in [n] \ \mr{and} \ k \in \Z \]
as in \cite[Equation~1.1]{FFR2017}. Here $e_1,\dots,e_n$ is the standard basis of $\C^n$. Since $U_n$ is required to be equal to $s_n U_0$, we loose no information, if we remove it from the parametrisation and impose the condition $\mr{pr}_{n}U_{n-1} \subset s_n U_0$. 
%-------------------------------------------------------------------------------------
\subsection{Finite Approximations by Quiver Grassmannians and Geometric Properties}
%-------------------------------------------------------------------------------------
E.~Feigin, M.~Finkelberg and M.~Reineke studied quiver Grassmannians for the loop quiver, to model finite approximations of the affine Grassmannian \cite{FFR2017}. Analogously, in this paper we restrict us to quiver Grassmannians for the equioriented cycle:
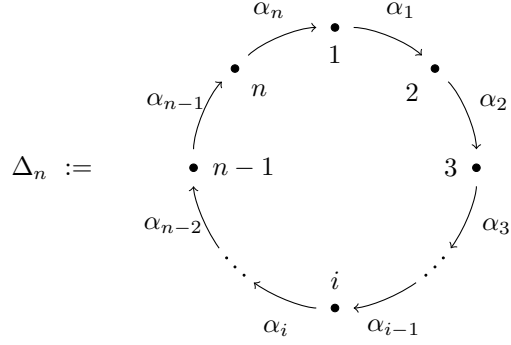
\begin{center}
\begin{tikzpicture}[scale=.62]
\node at ($(-6,0)$) {$\Delta_n \ := $};

% All nodes, node labels, and loops
\foreach \ang\lab\anch in {90/1/north, 45/2/{north east}, 0/3/east, 270/i/south, 180/{n-1}/west, 135/n/{north west}}{
  \draw[fill=black] ($(0,0)+(\ang:3)$) circle (.08);
  \node[anchor=\anch] at ($(0,0)+(\ang:2.8)$) {$\lab$};
}

% Top part of circle, arrows between different nodes and their labels
\foreach \ang\lab in {90/1,45/2,180/{n-1},135/n}{
  \draw[->,shorten <=7pt, shorten >=7pt] ($(0,0)+(\ang:3)$) arc (\ang:\ang-45:3);
  \node at ($(0,0)+(\ang-22.5:3.65)$) {$\alpha_{\lab}$};
}

% Bottom part of circle, arrows between different nodes and their labels
\draw[->,shorten <=7pt] ($(0,0)+(0:3)$) arc (360:325:3);
\draw[->,shorten >=7pt] ($(0,0)+(305:3)$) arc (305:270:3);
\draw[->,shorten <=7pt] ($(0,0)+(270:3)$) arc (270:235:3);
\draw[->,shorten >=7pt] ($(0,0)+(215:3)$) arc (215:180:3);
\node at ($(0,0)+(0-20:3.65)$) {$\alpha_3$};
\node at ($(0,0)+(315-25:3.65)$) {$\alpha_{i-1}$};
\node at ($(0,0)+(270-20:3.65)$) {$\alpha_i$};
\node at ($(0,0)+(225-25:3.65)$) {$\alpha_{n-2}$};

% Ellipsis
\foreach \ang in {310,315,320,220,225,230}{
  \draw[fill=black] ($(0,0)+(\ang:3)$) circle (.02);
}

\end{tikzpicture}
\end{center}
Both, the set of vertices, and arrows, of $\Delta_n$ are in bijection with the set $\Z_n:= \Z/n\Z$. 

For a positive integer $\omega$, the finite approximation of the degenerate affine flag variety is defined as
\[ \mathcal{F}l^a_\omega\big(\widehat{\mathfrak{gl}}_n\big)  := \Big\{   (U_k)_{k=0}^{n-1} \in \mathcal{F}l^a\big(\widehat{\mathfrak{gl}}_n\big) \ : \ V_{-\omega n} \subseteq U_0 \subseteq V_{\omega n}\Big\} .\] 
\begin{trm}
\label{trm:finite_approx-in-the-intro}
Let $\omega \in \N$ be given, define the quiver representation
\[ M_\omega := \Big( \, \big( V_i:= \C^{2\omega n} \big)_{i \in \Z_n}, \big( M_{\alpha_i} := s_1 \circ \mr{pr}_{\omega n} \big)_{i \in \Z_n} \,  \Big),\] and the dimension vector $\mb{e}_\omega := (e_i := \omega n)_{i \in \Z_n}$. Then the finite dimensional approximation of the degenerate affine flag variety is isomorphic to the quiver Grassmannian corresponding to $M_\omega$ and $\mb{e}_\omega$, i.e.
\[ \mathcal{F}l^a_\omega\big(\widehat{\mathfrak{gl}}_n\big) \cong \mr{Gr}^{\Delta_n}_{\mb{e}_\omega}\big(M_\omega\big). \]
\end{trm}
This construction allows us to obtain statements about geometric properties of the approximations from the corresponding quiver Grassmannians.
\begin{trm}
\label{trm:geom-prop-aff-flag-in-the-intro}
For $\omega \in \N$, the approximation $\mathcal{F}l^{a}_{\omega}\big(\widehat{\mathfrak{gl}}_n\big)$ of the degenerate affine flag variety satisfies:
\vskip3pt
(1) $\ \ $ It is a projective variety of dimension $\omega n^2$.
\vskip3pt
(2) $\ \ $ Its irreducible components are equidimensional.
\vskip3pt
(3) $\ \ $ It admits a cellular decomposition.
\vskip3pt
(4) $\ \ $ The irreducible components are normal, Cohen-Macaulay
\vskip1pt $\quad \quad \ $ and have rational singularities.
\vskip3pt
(5) $\ \ $ There is a bijection between irreducible components
\vskip1pt $\quad \quad \ $ and grand Motzkin paths of length $n$. 
\end{trm}
Grand Motzkin paths of length $n$ are lattice paths from $(0,0)$ to $(n,0)$ with steps $(1,1)$, $(1,0)$ and $(1,-1)$,  without the requirement that the path is not allowed to cross the $x$-axis. Accordingly the number of irreducible components is independent of the parameter $\omega$ and the same in every approximation. 

In Section~\ref{sec:quiver-grass-cycle}, we study geometric properties of quiver Grassmannians for the equioriented cycle. This specialises to the proof of Theorem~\ref{trm:geom-prop-aff-flag-in-the-intro} as given in Section~\ref{sec:proof-of-geom-prop}. The identification of cells and affine Dellac configurations is proven in Section~\ref{sec:aff-dellac}. In Section~\ref{sec:lin-deg-flag}, we generalise the constructions from the previous sections to certain linear degenerations of the affine flag variety.% and examine their geometric properties.

%-------------------------------------------------------------------------------------
\subsection{Isomorphism of Finite Approximations and Quiver Grassmannians}\label{sec:iso:fin-approx-quiv-grass}
%-------------------------------------------------------------------------------------
\begin{proof}[Proof of Theorem~\ref{trm:finite_approx-in-the-intro}]
In the approximation to the parameter $\omega$, the cyclic relations of the vector spaces, describing a point $(U_k)_{k=0}^{n-1}$ in the degenerate affine flag variety, induce the restrictions
\[  V_{-n\omega + k} \subseteq U_k \subseteq  V_{n\omega+k} \quad \mr{for} \ U_k \in \mr{SGr}_k. \]
Accordingly the corresponding approximations of the Sato Grassmannians 
\[ \mr{SGr}_{k,\omega} := \big\{ U \in \mr{SGr}_k : V_{-n\omega + k} \subseteq U \subseteq  V_{n\omega+k} \big\}\]
are isomorphic to the Grassmannian of vector subspaces $\mr{Gr}_{\omega n}(2\omega n)$. Here we identify the standard basis of $\C^{2 \omega n}$ with basis vectors of $V$ as
\[ e_j \longleftrightarrow v_{\omega n+k-j+1} \ \ \mathrm{for} \ \mathrm{all} \ j \in [2 \omega n].\]
Moreover, the Sato Grassmannian $\mr{SGr}_k$ is just the $k$-th shift of the Sato Grassmannian $\mr{SGr}_0$, i.e. $\mr{SGr}_k \cong s_k \mr{SGr}_0$ \cite[Section~1.2]{FFR2017}. Instead of one index shift by $n$ we can take the Sato Grassmannians $\mathrm{SGr}_0$ in the place of $\mathrm{SGr}_k$ and have a shift by one along each arrow of the quiver and for the maps $\mr{pr}_{k+1}$ we have to write $\mr{pr}_{1}$. If we want to apply these shifts on the left hand side of the containment relations, they become a shift by minus one. This identification works in the same way on the level of approximations. We obtain the projection $\mr{pr}_{\omega n}: \C^{2 \omega n} \to \C^{2 \omega n}$ if we apply the above base change for $k=0$ to the restriction of the projection $\mr{pr}_{1}: V \to V$ to the finite support induced by the approximation. Due to this base change, the index shift by minus one is turned into a shift by plus one. Combining these properties, we obtain the following chain of isomorphisms:
\begin{align*}
 \mathcal{F}l^a_\omega\big(\widehat{\mathfrak{gl}}_n\big)  &= \Bigg\{ \big(U_k \big)_{k=0}^{n-1} \in \prod_{k=0}^{n-1} \mathrm{SGr}_{k,\omega} \  : \ \mr{pr}_{k+1} U_k \subset U_{k+1}, \   \mr{pr}_{n}U_{n-1} \subset s_n U_0 \Bigg\}\\
  &\cong  \Bigg\{ \big(U_i \big)_{i \in \Z_n} \in \prod_{i \in \Z_n} \mathrm{SGr}_{0,\omega} \  : \ s_{-1} \circ \mr{pr}_{1} U_i \subset U_{i+1} \fa i \in \Z_n \Bigg\}\\
 &\cong \Bigg\{ \big(U_i \big)_{i \in \Z_n} \in \prod_{i \in \Z_n} \mathrm{Gr}_{\omega n}(2 \omega n) \  : \ s_{1} \circ \mr{pr}_{\omega n} U_i \subset U_{i+1} \fa i \in \Z_n \Bigg\}\\
 &\cong \mr{Gr}^{\Delta_n}_{\mb{e}_\omega}\big(M_\omega\big).
\end{align*} 
\end{proof}

%-------------------------------------------------------------------------------------
\section{Quiver Grassmannians for the Equioriented Cycle}\label{sec:quiver-grass-cycle}
%-------------------------------------------------------------------------------------
For every $i \in \Z_n$, we define the path with $\ell$ arrows starting at vertex $i$ as 
\[ p_i(\ell):=(i \vert \alpha_i \alpha_{i+1} \dots \alpha_{i+\ell-1} \vert i + \ell ).\]
The path algebra $\Kbb \Delta_n$ is denoted by $A_n$. This algebra is not finite dimensional because there are paths $p_i(\ell)$ of arbitrary length around the cycle. Let 
\[\mr{I}_N := \langle \, p_i(N) : i \in \Z_n \, \rangle \subset \Kbb \Delta_n \]
be the ideal of the path algebra generated by all paths of length $N$. For $N \in \N$, we define the bounded path algebra $A_ {n,N} :=  \Kbb \Delta_n / \mr{I}_N$. The following result is a special case of \cite[Theorem 5.4]{Schiffler2014}.
\begin{prop}
The category $\mr{rep}_\Kbb(\Delta_n,\mr{I}_N)$ of bounded quiver representations is equivalent to the category $A_{n,N}$-$\mr{mod}$ of (right) modules over the bounded path algebra.
\end{prop}
Let $P_i \in \mr{rep}_\Kbb(\Delta_n,\mr{I}_N)$ be the projective bounded representation of $\Delta_n$ at vertex $i \in \Z_n$. Define the projective representation 
\[ X := \bigoplus_{i \in \Z_n} P_i \otimes \Kbb^{x_i}, \]
where $x_i \in \Z_{\geq 0}$ for all $i \in \Z_n$.
Analogously, let $I_j \in \mr{rep}_\Kbb(\Delta_n,\mr{I}_N)$ be the injective bounded representation of $\Delta_n$ at vertex $j \in \Z_n$ and define the injective representation 
\[Y := \bigoplus_{j \in \Z_n} I_j \otimes \Kbb^{y_j},\]
with $y_j \in \Z_{\geq 0}$ for all $j \in \Z_n$. Throughout this section, we study quiver Grassmannians 
\(\mr{Gr}_{\mb{e}}^{\Delta_n}(X \oplus Y), \)
where $ \mb{e} := \bdim X$ is the dimension vector of $X$.

A representation $U$ of the cycle is called nilpotent if there exists an integer $N$ such that $U$ satisfies the relations in $\mr{I}_N$, i.e. cyclic concatenations of the maps corresponding to $U$ are zero after a certain length. The indecomposable nilpotent representations of $\Delta_n$ are parametrised by some starting vertex $i \in \Z_n$ and a length parameter $\ell \in \Z_{\geq 0}$ \cite[Theorem 7.6]{Kirillov2016}. We denote them by $U_i(\ell)$ and the simple representations over the vertices have length one, i.e. $S_i = U_i(1)$. 
It turns out that we can identify bounded projective and bounded injective representations of the cycle, via indecomposable nilpotent representations.
\begin{prop}
\label{prop:proj_inj_dual}
For $n,N \in \N$ and all $i,j \in \Z_n$ the projective and injective representations $P_i$ and $I_j$ of the bound quiver $(\Delta_n,\mr{I}_N)$ satisfy 
\[ P_i \cong U_{i}(N) \cong I_{i+N-1} \quad \mr{and} \quad  I_j \cong U_{j-N+1}(N) \cong P_{j-N+1}.\]
\end{prop}
\begin{proof}
It follows from the definition of projective and injective representations of bound quivers as in \cite[Definition 5.3]{Schiffler2014} that they are of the form as described in \cite[Theorem 7.6]{Kirillov2016}.  
\end{proof}
This identification allows us to apply Theorem~\ref{trm:framed-module} to the quiver Grassmannians as introduced above. Accordingly it is possible to realise them as framed module space and we can use Theorem~\ref{trm:group-action} to study their geometry.
%-------------------------------------------------------------------------------------
\subsection{Word Combinatorics}
%-------------------------------------------------------------------------------------
\begin{defi}
For a nilpotent representation $U_i(\ell)$ of the equioriented cycle on $n$ vertices, the corresponding \f{word} $w_i(\ell)$ is defined as
\[ w_i(\ell) := i \ \ \ i+1 \ \ \ i+2 \ \ \ \dots \ \ \ i+\ell-2 \ \ \ i+\ell-1 \]
where we view each number in $\Z_n$. 
\end{defi}
To each nilpotent quiver representation $X$ we assign a diagram $\vartheta_X$, consisting of the words corresponding to the indecomposable direct summands of $X$. Define $r_j(w)$ as the number of repetitions of the letter $j$ in a word $w$. These numbers can be used to compute the dimension of the space of morphisms between two indecomposable nilpotent representations of the cycle. The linearity of the morphisms allows us to generalise this formula to compute the dimension of the morphism space for all nilpotent representations of the cycle. Here we use the notation $[\, . \, , \, . \, ]$ as abbreviation for the dimension of the space of morphisms $\Hom_{\Delta_n}( \, . \, , \, . \, )$.
\begin{prop}
\label{prop:hom_dim_words}
For two indecomposable nilpotent representations $U_i(\ell)$ and $U_j(k)$ of $\Delta_n$, let $w_i(\ell)$ and $w_j(k)$ be the corresponding words. Then the dimension of the space of morphisms from $U_i(\ell)$ to $U_j(k)$ equals
\begin{align*}
 \big[ U_i(\ell), U_j(k) \big] = \min \big\{ r_i\big(w_j(k) \big), r_{j+k-1}\big(w_i(\ell) \big) \big\} = \ r_{j+k-1}\big(w_i(m) \big). 
\end{align*}
where $m:= \min\{ \ell,k\}$.
\end{prop}

This proposition is just a different formulation of a result by A.~Hubery \cite[Theorem 16 (1)]{Hubery2010}. It is also possible to compute the dimension of the space of morphisms by counting certain repetitions of the letter $i$ in the word corresponding to $U_j(k)$. Here we have to exclude the repetitions coming before $\max\{0, m-\ell\}$ such that the parametrisation of the word, wherein we have to count the repetitions, becomes more complicated. Hence we exclude this case from the proposition.

We can use word combinatorics to compute the dimension of the space of morphisms from an arbitrary representation in $\mr{R}_\mathbf{e}(\Delta_n,\mr{I}_N)$ to an indecomposable representation of maximal length.
\begin{prop}
\label{prop:top_hom_dim}
Let $M \in \mr{R}_\mathbf{e}(\Delta_n,\mr{I}_N)$. Then
$[M,U_i(N)] = e_{i+N-1}$ for all $i \in  \Z_n$.
\end{prop}
\begin{rem}
Here it is important that the indecomposable representation $U_i(N)$ is of maximal length, or at least longer than every summand of $M$. Otherwise, $M$ could contain $U_{i-1}(N+1)$, which contradicts the statement of the proposition because
\( [U_{i-1}(N+1),U_i(N)] = 0 \)
by Proposition~\ref{prop:hom_dim_words}, since $r_{i+N}\big(w_i(N) \big) = 0$. For the injective labelling of the indecomposable representations $U(j;N) := U_{j-N+1}(N)$, the statement of the proposition reads as 
\( [M,U(j;N)] = e_{j} \ \mr{for} \ \mr{all} \ j \in  \Z_n.  \)
\end{rem}
\begin{proof}[Proof of Proposition \ref{prop:top_hom_dim}]
Every representation $M \in \mr{R}_\mathbf{e}(\Delta_n,\mr{I}_N)$ is nilpotent and can be written as a direct sum of indecomposable nilpotent representations, namely
\[ M \cong \bigoplus_{j \in\Z_n} \bigoplus_{\ell \in [N]} U_j(\ell) \otimes \Kbb^{m_{j,\ell}}. \]
Thus 
\[ [M,U_i(N)] = \sum_{j \in\Z_n} \sum_{\ell \in [N]} m_{j,\ell} \cdot \dim \, [U_j(\ell),U_i(N)] \]
because the morphisms of quiver representations are linear maps between finite dimensional vector spaces. 

Now it suffices to show that \( [U_j(\ell),U_i(N)] = d_{i+N-1}\), where $\mathbf{d}:= \bdim U_j(\ell)$. We want to apply Proposition~\ref{prop:hom_dim_words}. By assumption we know that $\ell \leq N$, and hence we have to count the repetitions of the vertex $i+N-1$ (which is the end point of $U_i(N)$) in the word $w_j(\ell)$ corresponding to $U_j(\ell)$. This value is given by the ${(i+N-1)}$-th entry of the dimension vector of $U_j(\ell)$.  
\end{proof}
By \cite[Lemma 2.4]{CFR2012}, the dimension of the stratum of $U \in \mr{Gr}^{\Delta_n}_{\mb{e}}(X \oplus Y)$ is
\[ \dim \mathcal{S}_U = [ U, X \oplus Y ]- [ U,U].\] For the first part, we can apply Proposition~\ref{prop:top_hom_dim}, and obtain
\( [U, X \oplus Y ] = [X, X \oplus Y ] \)
for all $U \in \mr{Gr}^{\Delta_n}_{\mb{e}}(X \oplus Y)$, because $X$ and $Y$ consist of summands of the form $U_i(N)$, and the morphisms are linear. To compute the dimension of the quiver Grassmannian, we are interested in the value of  
\( [U,U], \)
and want to find the elements of the quiver Grassmannian minimising it. On the variety of quiver representations $\mr{R}_\mathbf{e}(\Delta_n,\mr{I}_N)$, we have an action of the group $\mr{G}_\mb{e}$. The dimension of the $\mr{G}_\mb{e}$-orbit of $U \in \mr{R}_\mathbf{e}(\Delta_n,\mr{I}_N)$ is computed as 
\( \dim \, \mr{G}_\mb{e}.U = \dim \, \mr{G}_\mb{e} - [U,U].\) 
Thus we can compute maximisers for the dimension of the $\mr{G}_\mb{e}$-orbit in $\mr{R}_\mathbf{e}(\Delta_n,\mr{I}_N)$, among the elements of the Grassmannian $\mr{Gr}^{\Delta_n}_{\mb{e}}(X \oplus Y)$, in order to find the strata of highest dimension in this quiver Grassmannian.
\begin{rem}
Let $Q$ be a Dynkin quiver and let $X$, $Y$ be exceptional representations of $Q$ such that $\mr{Ext}^1_Q(X,Y)=0$. Then
\( [ U, U ] \geq [ X, X ] \) 
holds for all $U \in \mr{Gr}^Q_{\mb{e}}(X \oplus Y)$, i.e. the dimension of the orbit of $X$ in $\mr{R}_\mathbf{e}(Q)$ is maximal among all elements of the quiver Grassmannian. Thus we obtain
\( \dim \mr{Gr}_{\mb{e}}^Q(X \oplus Y) = [ X, Y ], \)
and additionally the quiver Grassmannian is the closure of the stratum of $X$. For more details on this see \cite[Section~3.1]{CFR2012}.
\end{rem}
Unfortunately, this does not work in the same way for the equioriented cycle. But in some special cases at least the dimension formula holds. 
\begin{prop}
\label{prop:hom_ineq}
Let $N = \omega \cdot n$ for $\omega \in \N$. Then
\( [ U, U ] \geq [ X, X ] \)
holds for all $U \in \mr{Gr}^{\Delta_n}_{\mb{e}}(X \oplus Y)$.
\end{prop} 
There are counter examples for this result if $N \neq \omega \cdot n$. For instance let $n=4, N=5$ and $X= U_3(5) \oplus U_4(5)$, $Y= U_1(5) \oplus U_1(5)$, $U= U_2(4) \oplus U_2(4) \oplus U_3(2)$. Then $U \in \mr{Gr}^{\Delta_n}_{\mb{e}}(X \oplus Y)$ for $\mb{e}:= \bdim X$ and $[U,U]=5 < 6 = [X,X]$, which is computed with the word combinatorics as above. The indecomposable representation $U_i(\omega n)$ has length $\omega n$, which means that it is winding around the cycle on $n$ points exactly $\omega$ times. For this reason we refer to $\omega$ as winding number. The proof of this statement is based on the characterisation of minimal degenerations of orbits in the variety of quiver representations as given by G.~Kempken in \cite{Kempken1982}. 
\begin{proof}[Proof of Proposition~\ref{prop:hom_ineq}]
The idea of the proof is to show that for every element $U \in \mr{Gr}^{\Delta_n}_{\mb{e}}(X \oplus Y)$, its $\mr{G}_\mb{e}$-orbit in the variety of quiver representations is the degeneration of an orbit with the same codimension as the $\mr{G}_\mb{e}$-orbit of $X$, or already has the same codimension. 
In the first case, if the $\mr{G}_\mb{e}$-orbit of $U$ is contained in the closure of a different $\mr{G}_\mb{e}$-orbit, the codimension of $\mathcal{O}_U := \mr{G}_\mb{e}.U$ is strictly bigger. The following property of a subrepresentation helps us to decide how far the representation is degenerate from $X$. For $U \in \mr{Gr}_{\mb{e}}^{\Delta_n}(X \oplus Y)$ define
\[ S(U) := \big\{ U_i(\ell) \subseteq U \ \mr{direct} \ \mr{summand} \ : \ell < N \big\}. \]
This set includes all direct summands of $U$, which are not of maximal length. If the set $S(U)$ is empty, we directly obtain 
\([ U, U ] = [ X, X ], \)
since 
\( [ U_i(N),U_j(N) ] =\omega \)
for all $i,j \in \Z_n$ and $N= \omega \cdot n$. This follows from Proposition~\ref{prop:top_hom_dim} if we set $M = U_i(N)$. 

Now let $U \in \mr{Gr}_{\mb{e}}^{\Delta_n}(X \oplus Y)$ be given such that $S(U) \ne \emptyset$. Since all entries of the dimension vector $\bdim U$ are equal, and $U$ consists of indecomposable summands of lengths at most $N$, the set $S(U)$ has to contain at least two elements.

We can find $U_i(\ell), U_j(k) \in S(U)$ and can assume without loss of generality that $j$ is contained in the word $w_i(\ell+1)$. This pair has to exist because otherwise the dimension vector of $U$ could not be homogeneous, i.e. all entries being equal. By changing the labelling of the two representations, we can ensure that they satisfy the relation we want. Let $w_2$ be the overlap of the words $w_i(\ell)$ and $w_j(k)$. We can write them as $w_i(\ell) =w_1w_2$ and $w_j(k) = w_2w_3$, where it is possible that $w_2$ is the empty word. We define the representation
\[ \hat{U} := U \oplus U_i(i-j+\ell+k) \oplus U_j(j-i) \setminus U_i(\ell) \setminus U_j(k). \]
Here $\setminus \, U_i(\ell)$ means that we do not take the direct summand $U_i(\ell)$ of $U$ for the definition of $\hat{U}$. We have to show that $U$ is a degeneration of $\hat{U}$ and that $\hat{U}$ is contained in the quiver Grassmannian.

In the terminology of words the representation $U_i(i-j+\ell+k)$ corresponds to $w_1w_2w_3$ and $U_j(j-i)$ corresponds to $w_2$. We can assume that the word $w_1w_2w_3$ has not more than $N$ letters, because there have to exist  $w_i(\ell) =w_1w_2$ and $w_j(k) = w_2w_3$ such that this is satisfied. Without such words it would not be possible that the dimension vector of $U$ is homogeneous and that all words corresponding to it have at most $N$ letters.

By construction $U$ and $\hat{U}$ have the same dimension vector. It is sufficient to find an arbitrary embedding of $\hat{U}$ into $X \oplus Y$, to show that $\hat{U}$ is contained in the quiver Grassmannian. The easiest way to do this is to identify a segment wise embedding, i.e. $\iota : U(i;\ell) \hookrightarrow U(i;k)$ for $k \geq \ell$. For nilpotent representations of the equioriented cycle, the existence of segment wise embeddings is equivalent to the existence of arbitrary embeddings. This follows from the structure of the Auslander Reiten quiver for the indecomposable nilpotent representations of the equioriented cycle. 

For example, this is computed using the knitting algorithm \cite[Chapter 3.1.1]{Schiffler2014}. Given a fixed nilpotence parameter $N$, we take an equioriented type A quiver of length $2Nn$ and identify the vertex $i$ with all repetitions $i+kn$. Similarly we identify vertices in the Auslander Reiten quiver of type A and obtain the Auslander Reiten quiver for the equioriented cycle. 

Hence there exists a segment wise embedding of $U$ into $X \oplus Y$. From the structure of segment wise embeddings of indecomposable representations of the cycle as introduced above, we know that $U_i(i-j+\ell+k)$ embeds into the same $U_p(N)$ as $U_j(k)$, and $U_j(j-i)$ embeds into the same $U_q(N)$ as $U_i(\ell)$. It follows that the representation $\hat{U}$ embeds into the same summands of $X \oplus Y$ as $U$. 

Following \cite[Satz~5.5]{Kempken1982} by G.~Kempken, the orbit of $U$ is a degeneration of the orbit $\mathcal{O}_{\hat{U}}$, i.e. $ \mathcal{O}_U \subset \overline{\mathcal{O}_{\hat{U}}}$. Hence we obtain $\dim \mathcal{O}_U < \dim \mathcal{O}_{\hat{U}}$, since $U$ and $\hat{U}$ are not isomorphic because their diagrams of words are constructed such that they do not contain the same words. This degeneration might not be minimal, but here it is not of interest to find minimal degenerations. Thus we do not have to satisfy the restrictions on the words in \cite[Satz~5.5]{Kempken1982}. 

Since all the vector spaces $U_i$ for $i \in \Z_n$ corresponding to the subrepresentation $U$ are equidimensional, we can apply this procedure starting from any $U$ in the quiver Grassmannian, until we arrive at an $\hat{U} \in \mr{Gr}_{\mb{e}}^{\Delta_n}(X \oplus Y)$ with $S(\hat{U}) = \emptyset$. Thus we obtain 
\( [ U, U ] > [ X, X ], \)
for every $U \in \mr{Gr}_{\mb{e}}^{\Delta_n}(X \oplus Y)$ with $S(U) \ne \emptyset$.  
\end{proof}
For the proof it is crucial that the dimension vector of the subrepresentations is homogeneous, and that the length of the cycle divides the length of the indecomposable projective and injective representations. This is guaranteed by the condition $N= \omega n$. Otherwise we can not assure that the gluing procedure of the words ends in a representation with $S(U) = \emptyset$. In the setting where $N \neq \omega n$, it is not possible to control the minimal codimension of the $\mr{G}_\mb{e}$-orbits.
%-------------------------------------------------------------------------------------
\subsection{Dimension Formula and Parametrisation of Irreducible Components}
%-------------------------------------------------------------------------------------
For this subsection, we restrict us to the case $N= \omega n$. The bounded projective and injective representations in $ \mr{rep}_\Kbb(\Delta_n,\mr{I}_{ \omega n})$ will be denoted by $P_i^\omega$ and $I_j^\omega$. Based on Proposition~\ref{prop:hom_ineq} we can compute the dimension of $\mr{Gr}_{\mb{e}}^{\Delta_n}(X \oplus Y)$. In its proof the subrepresentations $U$, with the same codimension as $X$, are characterised. This allows to determine the irreducible components of the quiver Grassmannian.
\begin{lma}
\label{lma:dim_Gr}
Let 
\[ X_\omega := \bigoplus_{i \in \Z_n} P_i^\omega \otimes \Kbb^{x_i}  \ \mr{and} \  Y_\omega := \bigoplus_{j \in \Z_n} I_j^\omega \otimes \Kbb^{y_j}\]
and set $\mb{e}_\omega := \bdim X_\omega$, where $x_i, y_j \in \N$ for all $i,j \in \Z_n$ . The dimension of the quiver Grassmannian is computed as
\[ \dim \mr{Gr}^{\Delta_n}_{\mb{e}_\omega}(X_\omega \oplus Y_ \omega) = \omega k(m-k),\]
where $ k := \sum_{i \in \Z_n} x_i$ and $ m := \sum_{i \in Z_n} x_i + y_i$.
\end{lma}
\begin{proof}
For better legibility we drop the index $\omega$ in this proof, i.e. $\mb{e} := \mb{e}_\omega$, $X:=X_\omega$ and $Y:=Y_\omega$. By \cite[Lemma 2.4]{CFR2012} we have $\dim \mathcal{S}_{U} = [ U, X \oplus Y ]-[ X, X ]$ for all $U \in \mr{Gr}^{\Delta_n}_{\mb{e}}(X \oplus Y)$. From Proposition~\ref{prop:hom_ineq} we get \( [U, X \oplus Y ] = [X, X \oplus Y ] \) and Proposition~\ref{prop:top_hom_dim} yields \( [ U, U ] \geq [ X, X ] \). It follows that
 \(  \dim \mathcal{S}_U \leq \dim \mathcal{S}_{X} \)
holds for all $U \in \mr{Gr}^{\Delta_n}_{\mb{e}}(X \oplus Y)$. It remains to compute the dimension of the stratum of $X$, which is given by \( \dim \mathcal{S}_{X} = [ X, X \oplus Y ] -[ X, X ] = [ X, Y ]. \)

As in Proposition~\ref{prop:proj_inj_dual} we identify $P_i^\omega \cong U_i(\omega n)$ and $I_j^\omega \cong U_{j-\omega n+1}(\omega n)$, and apply Proposition~\ref{prop:top_hom_dim} to obtain 
\( [ U_i(\omega n),U_j( \omega n) ] = d_{j+ \omega n -1} = \omega \) 
for all $i,j \in \Z_n$ where $\mb{d} := \bdim U_i(\omega n)$. The dimension of the stratum computes as
\begin{align*} 
[ X, Y ] &= \sum_{i \in\Z_n} \sum_{j \in\Z_n} [P_i^\omega \otimes \Kbb^{x_i} ,I_j^\omega \otimes \Kbb^{y_j} ] = \sum_{i \in\Z_n} {x_i} \sum_{j \in\Z_n} {y_j}  [P_i^\omega  ,I_j^\omega  ]\\
&= \sum_{i \in\Z_n} {x_i} \sum_{j \in\Z_n} {y_j} \cdot \omega = \sum_{i \in\Z_n} {x_i} \cdot \omega  \cdot (m-k) =\omega \cdot k(m-k).
\end{align*}  
\end{proof}
Based on the characterisation of the strata with the same codimension as the stratum of $X_\omega$, from the previous section, we obtain the following parametrisation of the irreducible components of the quiver Grassmannians.
\begin{lma}
\label{lma:irr_comp}
The irreducible components of $\mr{Gr}^{\Delta_n}_{\mb{e}_\omega}(X_\omega \oplus Y_ \omega)$ are in bijection with the set
\[ C_k(\mb{d}) := \Big\{ \mb{p} \in \Z_{\geq 0}^{n} : p_i \leq d_i \ \mr{for} \ \mr{all} \ i \in \Z_n, \sum_{i \in \Z_n} p_i = k \Big\}, \]
where $d_i := y_i + x_{i+1}$ and they all have dimension $\omega k(m-k)$. 
\end{lma}
\begin{rem}
In particular, the number of irreducible components is independent of the winding number $\omega$.
\end{rem}
\begin{proof}
We use the interpretation of the Grassmannian as framed moduli space 
\[ \mr{Gr}_{\mb{e}_\omega}^{\Delta_n}(X_\omega \oplus Y_\omega) \cong \mr{R}_{\mb{e}_\omega,\mb{d}}^s(\Delta_n,\mr{I}_{\omega n})/\mr{G}_{\mb{e}_\omega} \]
as in Theorem~\ref{trm:framed-module}. Here the $i$-th entry of $\mb{d}$ is given by the multiplicity of the injective representation $I_i^\omega$ as summand of $X_\omega \oplus Y_\omega$, and these numbers are independent of the winding number $\omega$. 

The irreducible components of $\mr{Gr}_{\mb{e}_\omega}^{\Delta_n}(X_\omega \oplus Y_\omega)$ are given by the closures of the strata which are maximal in the partial order $\mathcal{S}_U \leq \mathcal{S}_V  :\Leftrightarrow \mathcal{S}_U \subseteq \overline{\mathcal{S}_V}$. We apply Theorem~\ref{trm:group-action} to this setting such that the irreducible components are in bijection with the maximal elements of $\mr{R}_{\mb{e}_\omega}^{(\mb{d})}(\Delta_n,\mr{I}_{\omega n})/\mr{G}_{\mb{e}_\omega}$, with respect to the partial order induced by the inclusion of orbits in orbit closures. 

In the proof of Proposition~\ref{prop:hom_ineq}, we have seen that the maximal elements for this order are parametrised by the $U \in \mr{Gr}_{\mb{e}_\omega}^{\Delta_n}(X_\omega \oplus Y_\omega)$ such that $S(U) = \emptyset$, i.e. all summands of $U$ have the same dimension vector, and are of the form $U_i(\omega n)$ for some $i \in\Z_n$. Accordingly, these subrepresentations are representatives for strata whose closures give the irreducible components. The set of all such subrepresentations of $X_\omega \oplus Y_\omega$, with dimension vector $\mb{e}_\omega$, is parametrised by the set $C_k(\mb{d})$. 
For every tuple $\mb{p} \in C_k(\mb{d})$, define the representation
\[ U(\mb{p}) := \bigoplus_{i \in \Z_n} U(i;N)\otimes \Kbb^{p_i}.\] 
The assumption \(\sum_{i \in \Z_n} p_i = k \) ensures that the dimension vector of $U(\mb{p})$ is $\mb{e}$, since the dimension vector of $U(i;N)$ has all entries equal to $\omega$. The restriction
\[p_i  \in \{0,1,\dots, d_i \} \ \mr{for} \ \mr{all} \ i \in \Z_n\]
is necessary to guarantee that the representation $U(\mb{p})$ corresponding to the tuple $\mb{p}$ embeds into $X_\omega \oplus Y_\omega$. The dimension of the irreducible components is computed in the same way as done for the stratum $ \mathcal{S}_{X_\omega}$ in Lemma~\ref{lma:dim_Gr}. 
\end{proof}
For arbitrary $N$, we would have $d_i := y_i + x_{i-N+1}$ but for $N=\omega \cdot n$ the shift by $N$ does not change the index, because it is considered as a number in $\Z_n$.
The number of irreducible components is bounded by $\binom{n+k-1}{k}$ since 
\[ C_k(\mb{d})\subseteq C_k := \Big\{ \mb{p} \in \Z_{\geq 0}^{n} : p_i \leq k \ \mr{for} \ \mr{all} \ i \in \Z_n, \sum_{i \in Z_n} p_i = k \Big\}. \]
The set $C_k$ is the set of all partitions of the number $k$ into at most $n$ parts. Partitioning the number $k$ into at most $n$ parts, is equivalent to choosing $n-1$ points, out of $n+k-1$ points, to be the separators between the $n$ parts of a partition of the remaining $k$ points. The number of these choices is given by 
\[ \vert C_k \vert = \binom{n+k-1}{n-1} = \binom{n+k-1}{k}.\]
For the case $N=n$, this parametrisation of the irreducible components, together with a precise count, is proven in the thesis of N.~Haupt \cite[Proposition 3.6.16]{Haupt2011}. We can use his formula for the case $\omega =1$, to compute the number of irreducible components of the quiver Grassmannian 
\[ \mr{Gr}_{\mb{e}_\omega}^{\Delta_n}\big(X_\omega \oplus Y_\omega\big), \]
because the number of irreducible components is independent of $\omega$.
%-------------------------------------------------------------------------------------
\subsection{Geometric Properties}\label{sec:geom-prop}
%-------------------------------------------------------------------------------------
Back in the setting where the indecomposable summands of $X$ and $Y$ have arbitrary but all the same length $N$, we do not have a parametrisation of the irreducible components, but nevertheless they have the following properties.
\begin{lma}
\label{lma:rat_sing}
The irreducible components of $\mr{Gr}^{\Delta_n}_{\mb{e}}(X \oplus Y)$ are normal,\\ Cohen-Macaulay and have rational singularities.
\end{lma}
\begin{proof}
In her thesis G.~Kempken shows that the orbit closures inside $\mr{R}_\mb{e}(\Delta_n)$ are normal, Cohen-Macaulay and have rational singularities (compare \cite[Satz 4.5]{Kempken1982} combined with \cite[p. 50]{Kempf1973}). Her result holds for arbitrary representations of $\Delta_n$. We can apply it to orbit closures of nilpotent representations in $\mr{R}_\mb{e}(\Delta_n,\mr{I}_N)$, because by \cite[Korollar 2.10]{Kempken1982} there are no non-nilpotent representations inside these orbit closures. Combining this with Theorem~\ref{trm:group-action}, we get that the closures of the strata in the quiver Grassmannian have rational singularities, which again combined with \cite[p. 50]{Kempf1973} yields that they are normal and Cohen-Macaulay. Applying it to the closures of the strata, which are maximal for the partial order as introduced in Lemma~\ref{lma:irr_comp}, we obtain the desired result. 
\end{proof}
Moreover, G.~Kempken gives a description of the types of singularities which can occur, and she also describes the structure of the orbit closures and the codimension of the minimal degenerations of orbits.
%-------------------------------------------------------------------------------------
\subsection{Torus Action and Cellular Decomposition}\label{sec:cell-decomp}
%-------------------------------------------------------------------------------------
From now on, we restrict us to the setting where $\Kbb=\C$. For every nilpotent representation $U \in \mr{rep}_{\C}(\Delta_n)$, there exists a nilpotence parameter $N \in \N$ such that $U \in \mr{rep}_{\C}(\Delta_n, \mr{I}_N)$. Hence by \cite[Theorem~1.11]{Kirillov2016} it is conjugated to a direct sum of indecomposable nilpotent representations, i.e.
\[ U \cong U(\mb{d}) := \bigoplus_{i \in \Z_n} \bigoplus_{\ell = 1}^N U(i;\ell)\otimes \C^{d_{i,\ell}}, \]
where $d_{i,\ell} \in \Z_{\geq 0}$ for all $i \in \Z_n$ and $\ell \in [N]$. 

Let $\mathcal{B}_i := \{ v_{1}^{(i)}, v_{2}^{(i)}, \dots , v_{m_i}^{(i)} \}$ be a basis of the vector space $V_i$ corresponding to the representation \( M:= ( \, (V_i)_{i \in \Z_n},(M_\alpha)_{\alpha \in \Z_n} \,) \cong U(\mb{d}) \) where $\mb{m} := \bdim M $. Here we choose $M$ such that $M$ and $U(\mb{d})$ have the same coefficient quiver. Hence $M$ is just a different parametrisation of the quiver representation $U(\mb{d})$. It is required to give a formal definition for the torus action. A \f{segment} of $M$ is a maximal collection of vectors $\{ v_{k}^{(i)}\} \subseteq \mathcal{B}_\bullet := \cup_{i \in \Z_n}\mathcal{B}_i$ such that there is a unique starting point $v_{k_0}^{(i)}$, and every other element $v_{k'}^{(j)}$ of the collection is computed as
\[ v_{k'}^{(j)} = M_{j-1} \circ \dots \circ M_{i+1} \circ M_i \Big(v_{k_0}^{(i)}\Big). \]
To simplify notation we also use the index $i$ for the arrow $\alpha : i \to i+1$. The segments of $M$ represent its direct summands $U_i(\ell)= U(i+\ell-1;\ell)$.

The number of indecomposable direct summands of $U(\mb{d})$, ending over the vertex $i \in \Z_n$, is given by 
\[ d_i := \sum_{\ell = 1}^N d_{i,\ell}. \]
We rearrange the segments of $M$ such that they end in the $d_i$ last basis vectors over each vertex. In each package of segments ending over some vertex we order them from long to short, i.e. the shortest segment ending over the $i$-th vertex ends in the basis vector $v_{m_i}^{(i)}$. We continue the segments, to the last free basis vectors over the $(i-1)$-th vertex of of $\Delta_n$, such that their initial order by length is preserved. This way of arranging the segments ensures that they do not cross. 

The condition $d(\alpha) := d_{s_\alpha}$ for all $\alpha \in \Z_n$ induces a grading of the vertices in the coefficient quiver of $U(\mb{d})$. Choose the nilpotence parameter $N$ such that there exists an $i \in \Z_n$ such that $d_{i,N} \neq 0$ and define $q_i := m_i - d_i$ for all $i \in \Z_n$. Take $i_0 \in \Z_n$ such that $d_{i_0,N} \geq d_{i,N}$ for all $i \in \Z_n$. Let $j_0 := i_0 -N+1 \mod n$.

Give weight one to $v_1^{(j_0)}$. Compute the weight of the other points on the segment starting in $v_1^{(j_0)}$ using the grading of the arrows as defined above. The end point of this segment is $v_{q_{i_0}+1}^{(i_0)}$ and its weight is \[w := 1 + \sum_{\ell =0}^{N-2} d(\alpha_{j_0+\ell}).\] For $t \in [d_{i_0}]$, the weight of $v_{q_{i_0}+j}^{(i_0)}$ is $w+t-1$. Let $r_i := m_i - d_{i,1}$ be the largest index such that $v_{r_{i}}^{(i)}$ lives on a non-trivial segment and let $k$ be the weight of $v_{r_{i_0}}^{(i_0)}$. If $r_{i_0}=0$, we set $k=0$, since it implies that $r_i=0$ for all $i \in \Z_n$.

For the remaining end points $v_{q_{i}+j}^{(i)}$ with $i \in \Z_n$ and $j \in [d_i]$ their weight is given by $k+j+d_{i,1}-d_i$. The weight of all other vertices is computed along the arrows in opposite direction. This grading is strictly increasing with the indices of basis vectors in each basis $\mathcal{B}_i$. Moreover it is compatible with the assumptions in \cite[Theorem~1]{Cerulli2011}.

Using this grading, we can define a torus action on the quiver representation $M$. An element $\lambda \in T:= \C^*$ acts on every element $b \in \mathcal{B}_\bullet$ of the new ordered basis as
\[ \lambda . b := \lambda^{d(b)} b. \]
By linearity, this action extends to all elements of $M$, and also to the quiver Grassmannian \cite[Lemma 1.1]{Cerulli2011}. This implies that the Euler-Poincar\'e characteristic of the quiver Grassmannian is given by the number of its torus fixed points \cite[Theorem~1]{Cerulli2011}, and that they are parametrised by successor closed subquivers \cite[Proposition 1]{Cerulli2011}. Analogous to \cite[Section~6.4]{CFFFR2017}, this action induces a cellular decomposition of the quiver Grassmannian. 

For a torus fixed point $L \in \mr{Gr}^{\Delta_n}_{\mb{e}}(M)^T$, its \f{attracting set} is defined as
\[ \mathcal{C}(L) := \Big\{ V \in \mr{Gr}^{\Delta_n}_{\mb{e}}(M) : \lim_{\lambda \to 0} \lambda . V = L \Big\} \]
\begin{trm}
\label{trm:cell_decomp-approx-lin-deg-aff-flag}
Let $M = U(\mb{d})$ and $\mb{e} \leq \bdim M$. For every $L \in \mr{Gr}^{\Delta_n}_{\mb{e}}(M)^T$, the subset $\mathcal{C}(L) \subseteq \mr{Gr}^{\Delta_n}_{\mb{e}}(M)$ is an affine space, and the quiver Grassmannian admits a cellular decomposition 
\[  \mr{Gr}^{\Delta_n}_{\mb{e}}(M) = \bigsqcup_{L \in \mr{Gr}^{\Delta_n}_{\mb{e}}(M)^T} \mathcal{C}(L).  \]
\end{trm}
In Particular, these quiver Grassmannians have property (S) \cite[Definition 1.7]{CLP1988}. For Dynkin qiuvers and acyclic orientations of affine Dynkin diagrams, the existence of a cellular decomposition is proven in \cite{CEFR2018}. Using similar methods, it might also be possible to prove it for the equioriented cycle. But this explicit parametrisation of the cells allows us to use the combinatorics of coefficient quivers to study geometric properties of the quiver Grassmannians, like Poincar\'e polynomials or moment graphs for the action of a bigger torus.	
\begin{proof}
By \cite[Lemma~4.12]{Carrell02}, there exists a total order of the fixed points 
$ \mr{Gr}^{\Delta_n}_{\mb{e}}(M)^T = \{L_1, \dots, L_r \}$ such that the decomposition as in the statement of the theorem is an $\alpha$-partition, i.e. 
\( \bigsqcup_{j=1}^s \mathcal{C}(L_j) \)
is closed in $\mr{Gr}^{\Delta_n}_{\mb{e}}(M)$ for all $s \in [r]$. Hence it remains to show that the $\mathcal{C}(L)$ are isomorphic to affine spaces.

It follows from the definition of quiver Grassmannians that
\[\mr{Gr}^{\Delta_n}_{\mb{e}}(M) = \Big\{ (V_i)_{i \in \Z_n} \in \prod_{i \in \Z_n} \mr{Gr}_{e_i}({m_i}) \ : \ M_iV_i \subseteq V_{i+1} \fa i \in \Z_n \Big\},\]
where $\mr{Gr}_{e}({m})$ is the Grassmannian of $e$-dimensional subspaces in $\C^{m}$. First we show that the attractive sets 
\(\mathcal{C}(L^{(i)}) := \mathcal{C}(L) \cap \mr{Gr}_{e_i}({m_i}) \) are isomorphic to affine spaces.

By the structure of the $\C^*$-fixed points as described in \cite[Theorem~1]{Cerulli2011}, there exists an index set $K_i \in \binom{[m_i]}{e_i}$ such that $L^{(i)}$ is the span of the $v^{(i)}_k$ for $k \in K_i$. From the structure of the $\C^*$-action we deduce that a point $V_i$ in the attracting set $\mathcal{C}(L^{(i)})$ has generators
\[  w_{k}^{(i)} = v_{k}^{(i)} + \sum_{\ell \in [m_i] \setminus  [k] \, : \ \ell \notin K_i} \mu_{\ell,k}^{(i)} v_\ell^{(i)}\]
for $k \in K_i$ and $\mu_{\ell,k}^{(i)} \in \C$. Hence $\mathcal{C}(L^{(i)})$ is an affine space. Observe that for a representation $V$ in an attracting set of a $\C^*$-fixed point $L$, it holds that
\begin{align*}
V \in \mathcal{C}(L) \Leftrightarrow  V \in  \mr{Gr}^{\Delta_n}_{\mb{e}}(M)  \cap \prod_{i \in\Z_n} \mathcal{C}(L^{(i)}) . 
\end{align*}

Now we describe the equations arising from the condition $M_iV_i \subseteq V_{i+1}$. By the arrangement of the segments in $Q(M)$, it follows that $M_i w_{k}^{(i)}=0$ if $M_i v_{k}^{(i)}=0$. In this case there are no relations. Assume $M_i v_{k}^{(i)} \neq 0$ and let $k' \in K_{i+1}$ be such that $M_i v_{k}^{(i)}=v_{k'}^{(i+1)}$. Analogously, we define the index set $K_i' \subseteq [m_{i+1}]$. Then $M_i w_{k}^{(i)}$ equals
\[  M_iv_{k}^{(i)} + \sum_{\substack{\ell\in [m_i] \, : \\ \ell \notin [k],  \\ \ell \notin K_i}} \mu_{\ell,k}^{(i)} M_iv_\ell^{(i)}= v_{k'}^{(i+1)} + \sum_{\substack{\ell \in [m_i] \setminus  [k] \, : \\M_iv_\ell^{(i)} \neq 0, \\ \ell' \notin K_{i+1}}} \mu_{\ell,k}^{(i)} \,v_{\ell'}^{(i+1)} +\sum_{\substack{\ell \in [m_i] \setminus  [k] \, : \\ M_iv_\ell^{(i)} \neq 0, \\ \ell' \in K_{i+1} \setminus K_i' }} \mu_{\ell,k}^{(i)} \,v_{\ell'}^{(i+1)}. \]
Since $Q(L)$ is successor closed, we have $K_i' \subseteq K_{i+1}$, as the indices of the end points of segments in $K_i$ play no role in the set $K_i'$. The vector $M_i w_{k}^{(i)}$ is included in the span of the $w_{q}^{(i+1)}$ for $q \in K_{i+1}$ if and only if it equals 
\begin{align*}
&w_{k'}^{(i+1)}+ \sum_{\substack{\ell \in [m_i] \setminus  [k] \, : \\ M_iv_\ell^{(i)} \neq 0, \\ \ell' \in K_{i+1} \setminus K_i'}} \mu_{\ell,k}^{(i)} \, w_{\ell'}^{(i+1)}\\
= \, &v_{k'}^{(i+1)}+ \sum_{\substack{\ell \in [m_i] \setminus  [k] \, :  \\ M_iv_\ell^{(i)} \neq 0, \\ \ell' \in K_{i+1} \setminus K_i'}} \mu_{\ell,k}^{(i)}  v_{\ell'}^{(i+1)} + \sum_{\substack{j \in [m_{i+1}]: \\ j \notin  [k'], \\  j \notin K_{i+1}}} \Big( \mu_{j,k'}^{(i+1)} + \sum_{\substack{\ell \in [m_i] \setminus  [k] \, :  \\ M_iv_\ell^{(i)} \neq 0, \, \ell' < j, \\ \ell' \in K_{i+1} \setminus K_i'}} \mu_{\ell,k}^{(i)} \mu_{j,\ell'}^{(i+1)} \Big) v_{j}^{(i+1)}.
\end{align*}
This is equivalent to the equations
\begin{align*}
\mu_{j,k}^{(i)} &= \mu_{j',k'}^{(i+1)} + \sum_{\substack{\ell \in [j-1] \setminus [k]\, :  \\ M_iv_\ell^{(i)} \neq 0, \\ \ell' \in K_{i+1} \setminus K_i' }} \mu_{\ell,k}^{(i)} \, \mu_{j',\ell'}^{(i+1)}  \quad \big( \mr{if} \ j \in [m_i]\setminus [k] \, : \ M_i v_{j}^{(i)}\neq 0, \ j' \notin K_{i+1}\big), \\
0 &= \mu_{h,k'}^{(i+1)} + \sum_{\substack{\ell \in [m_i] \setminus  [k] \, : \\ M_iv_\ell^{(i)} \neq 0, \ \ell' < h, \\ \ell' \in K_{i+1} \setminus K_i', }} \mu_{\ell,k}^{(i)} \, \mu_{h,\ell'}^{(i+1)} \quad \binom{\mr{if} \ h \in [m_{i+1}] \setminus  [k'] \, : \ h \notin K_{i+1},}{\nexists \ell \in [m_i]\setminus[k] \ \mr{s.t.:}  \ M_i v_{\ell}^{(i)}= v_h^{(i+1)}}.
\end{align*}
It remains to show that these equations parametrise an affine subspace in the product of Grassmannians $\mr{Gr}_{e_i}(m_i)$. Starting at the end points of the segments, we define a total order of the parameters $\mu_{\ell,k}^{(i)}$ and relabel them as $\mu_1, \dots \mu_q$. Then all of the equations above describe one $\mu_k$ in terms of monomials in the $\mu_\ell$ with $\ell \in [q] \setminus [k]$. 
\end{proof}
%-------------------------------------------------------------------------------------
\subsection{Proof of Theorem~\ref{trm:geom-prop-aff-flag-in-the-intro}}\label{sec:proof-of-geom-prop}
%-------------------------------------------------------------------------------------
The representation $M_\omega$ as defined in Theorem~\ref{trm:finite_approx-in-the-intro} is nilpotent, and every nilpotent representation of $\Delta_n$ has a decomposition into the indecomposable representations $U(i;\ell)$ \cite[Theorem~1.11]{Kirillov2016}. The following observation is the key to examine geometric properties of the degenerate affine flag, because it allows us to apply the results about quiver Grassmannians from this section.
\begin{lma}
For $\omega \in \N$ there is an isomorphism of $\Delta_n$-representations
\[ M_\omega := \Big( \, \big( V_i:= \C^{2\omega n} \big)_{i \in \Z_n}, \big( M_{\alpha_i} := s_1 \circ \mr{pr}_{\omega n} \big)_{i \in \Z_n} \,  \Big) \cong \bigoplus_{j \in \Z_n} I_j^\omega \otimes \C^2.\]
\end{lma}
\begin{proof}
The coefficient quiver of $M_\omega$ consists of $2n$ segments of length $\omega n$, and over each vertex there start and end exactly two segments. Each of these segments corresponds to a bounded injective representation $I_j^\omega \cong U(j;\omega n )$. 
\end{proof}
With this identification of quiver representations, Theorem~\ref{trm:geom-prop-aff-flag-in-the-intro} is a specialisation of the results about quiver Grassmannians for the cycle as developed above.
\begin{proof}[Proof of Theorem~\ref{trm:geom-prop-aff-flag-in-the-intro}]
Part (1) follows from Lemma~\ref{lma:dim_Gr} and the identification with a quiver Grassmannian as in Theorem~\ref{trm:finite_approx-in-the-intro}, i.e.
\[ \dim \mathcal{F}l^a_\omega\big(\widehat{\mathfrak{gl}}_n\big) = \omega n(2n - n) = \omega n^2.\] From the $\C^*$-action and the induced cellular decomposition as in Theorem~\ref{trm:cell_decomp-approx-lin-deg-aff-flag}, we obtain (3). With $x_i=y_i=1$, we obtain (2) and (5) as special case of Lemma~\ref{lma:irr_comp}, i.e. the irreducible components are equidimensional and parametrised by
\[ \Big\{ \mb{p} \in \Z_{\geq 0}^{n} : p_i \leq 2 \ \mr{for} \ \mr{all} \ i \in \Z_n, \sum_{i \in \Z_n} p_i = n \Big\}. \]
With $b_i := p_i-1$, the tuples $\mb{b}$ describe the steps in grand Motzkin paths, and we obtain a bijection between the above set parametrising the irreducible components and the set of grand Motzkin paths of length $n$. The geometric properties in part (4) are obtained as application of Lemma~\ref{lma:rat_sing}.  
\end{proof}
%-------------------------------------------------------------------------------------
\section{Affine Dellac Configurations}\label{sec:aff-dellac}
%-------------------------------------------------------------------------------------
For the Feigin degeneration of the classical flag variety of type $A_n$, the Poincar\'e polynomial can be computed using Dellac configurations, which are counted by the median Genocchi numbers. This description was develloped by E.~Feigin in \cite{Feigin2011}. The torus fixed points of the symplectic degenerated flag variety are identified with symplectic Dellac configurations by X.~Fang and G.~Fourier in \cite{FaFo2016}. In this section, we identify affine Dellac configurations with the cells of the degenerate affine flag variety, based on the parametrisation of its torus fixed points via successor closed subquivers. 
\begin{defi}
For $n \in \N$ an \textbf{affine Dellac configuration} $\hat{D}$ to the parameter $\omega \in \N$ consists of a rectangle of $2n\times n$ boxes, with $2n$ entries $k_j \in \{0,1,2,\dots,\omega\}$ such that:
\begin{enumerate}
\item $ \ $There is one number in each row.
\item $ \ $There are two numbers in each column.
\item $ \ $\(\sum_{j=1}^{2n}(p_j+nr_j) = \omega n^2,\)
\end{enumerate}
where $r_j:=\max \{k_j-1,0\}$. The left hand side and the right hand side of the rectangle are identified to obtain boxes on a cylinder. There is a staircase around the cylinder from left to right as separator. In the planar picture, we draw it from the lower left corner to the upper right corner of the rectangle of boxes. With respect to the staircase, $p_j$ is the number of steps from the separator to the entry going left. If the entry is zero, the position is zero as well. The set of \textbf{affine Dellac configurations} to the parameter $\omega$ is denoted by $\widehat{DC}_n(\omega)$.
\end{defi}
\begin{ex}
For $n=4$ and $\omega = 3$, the subsequent configuration
\begin{center}
\begin{tikzpicture}[scale=.5]
% the marking:
\node at (-0.5+3,0.5-1) {$2$};
\node at (-0.5+2,0.5-2) {$2$};
\node at (-0.5+4,0.5-3) {$3$};
\node at (-0.5+1,0.5-4) {$1$};
\node at (-0.5+2,0.5-5) {$2$};
\node at (-0.5+3,0.5-6) {$0$};
\node at (-0.5+4,0.5-7) {$3$};
\node at (-0.5+1,0.5-8) {$2$};

\node at (-0.5+4.2,0.5-9.5) {$\Sigma=$};

% the positions: 
\node at (-0.5+6,0.7) {$p_j:$};
\node at (-0.5+6,0.5-1) {$1$};
\node at (-0.5+6,0.5-2) {$1$};
\node at (-0.5+6,0.5-3) {$2$};
\node at (-0.5+6,0.5-4) {$4$};
\node at (-0.5+6,0.5-5) {$2$};
\node at (-0.5+6,0.5-6) {$0$};
\node at (-0.5+6,0.5-7) {$2$};
\node at (-0.5+6,0.5-8) {$4$};
\node at (-0.5+6,0.5-9.5) {$16$};

% the rest: 
\node at (-0.5+8,0.7) {$r_j:$};
\node at (-0.5+8,0.5-1) {$1$};
\node at (-0.5+8,0.5-2) {$1$};
\node at (-0.5+8,0.5-3) {$2$};
\node at (-0.5+8,0.5-4) {$0$};
\node at (-0.5+8,0.5-5) {$1$};
\node at (-0.5+8,0.5-6) {$0$};
\node at (-0.5+8,0.5-7) {$2$};
\node at (-0.5+8,0.5-8) {$1$};
\node at (-0.5+8,0.5-9.5) {$8$};

%the boxes
\foreach \n in{4}{
\foreach \i in {0,1,2,3,4}{
  \draw(\i,0) -- (\i,-\n-\n);
  \draw(0,-\i) -- (\n,-\i);
  \draw(0,-\i-\n) -- (\n,-\i-\n);  
}
}
%the seperator
\foreach \n in{4}{
\foreach \i in {0,1,2,3 }{
  \draw [line width=0.5mm, black] (\i,\i-\n+1) -- (\i+1,\i-\n+1);
  \draw [line width=0.5mm, black] (\i,\i-\n-\n+1) -- (\i+1,\i-\n-\n+1);
  \draw [line width=0.5mm, black] (\i,\i-\n) -- (\i,\i-\n+1);
  \draw [line width=0.5mm, black] (\i,\i-\n-\n) -- (\i,\i-\n-\n+1);
}
}
\end{tikzpicture}
\end{center}
is contained in the set $\widehat{DC}_4(3)$, since $16+4\cdot 8 = 48 = 3\cdot 4^2.$
\end{ex}

\begin{trm}
\label{trm:bij-cells-aff-delllac}
For $\omega \in \N$, the cells in the approximation $\mathcal{F}l^{a}_{\omega}\big(\widehat{\mathfrak{gl}}_n\big)$ of the degenerate affine flag variety are in bijection with affine Dellac configurations to the parameter $\omega$. 
\end{trm}
%-------------------------------------------------------------------------------------
\subsection{Proof of Theorem~\ref{trm:bij-cells-aff-delllac} } 
%-------------------------------------------------------------------------------------
For $k \in \N$, let $[k]_0 $ denote the set $\{0,1,\dots,k\}$. The combinatorics of successor closed subquivers yields the following parametrisation of the cells.
\begin{lma}
\label{lma:cyclic-dellac}
The cells of $\mathcal{F}l^{a}_{\omega}\big(\widehat{\mathfrak{gl}}_n\big)$ are parametrised by the elements of the set
\[ \mathcal{C}^{a}_n(\omega ) := \Big\{ \mb{l} := (\ell_{i,1},\ell_{i,2})_{i \in \Z_n} \in \bigoplus_{i \in \Z_n} [\omega n]_0 \times [\omega n]_0 : \ \bdim U(\mb{l}) = \mb{e}_{\omega} \Big\}\]
where $\mb{e}_{\omega} := (e_i=\omega n)_{i \in \Z_n}$ and
\[ U(\mb{l}) := \bigoplus_{i \in \Z_n} U(i;\ell_{i,1}) \oplus U(i;\ell_{i,2}). \]
\end{lma}
\begin{proof}
The approximation $\mathcal{F}l^{a}_{\omega}\big(\widehat{\mathfrak{gl}}_n\big)$ is isomorphic to the quiver Grassmannian
\[ \mr{Gr}^{\Delta_n}_{\mb{e}_{\omega}}\big(M_{\omega}^a\big), \]
where $\mb{e}_{\omega} :=  (e_i=\omega n)_{i \in \Z_n}$ and 
\(M_{\omega}^a = \bigoplus_{i \in \Z_n} U(i;\omega n) \otimes \C^2.\)

By \cite[Proposition~1]{Cerulli2011} and Theorem~\ref{trm:cell_decomp-approx-lin-deg-aff-flag} we obtain that the cells in the approximation of the degenerate flag variety are in bijection with successor closed subquivers in the coefficient quiver of $M_{\omega}^a$, which have $\omega n$ marked points over each vertex $i \in \Z_n$. Each of the subquivers consists of $2n$ segments with length between zero and $\omega n$. There are exactly two segments ending over each vertex and they correspond to the indecomposable representations $U(i;\ell_{i,1})$ and $U(i;\ell_{i,2})$, where $U(i;0)$ is the zero representation independent of the indexing vertex. The marked points of the subquiver are encoded in the dimension vector of the representation $U(\mb{l})$. Hence the set $\mathcal{C}^{a}\big(n,\omega \big)$ is in bijection with the set of successor closed subquivers, which parametrise the cells of $\mathcal{F}l^{a}_{\omega}\big(\widehat{\mathfrak{gl}}_n\big)$. 
\end{proof}
Now we want to study how every row of an affine Dellac configuration encodes an indecomposable representation $U(i;\ell_{i,k})$ for $k \in \{1,2\}$. Given an entry $k_j$ of a configuration, and its relative position $p_j$ to the separator, we set $\ell_{i,k} = p_j+ n r_j$, where $i=j$, $k=1$ for $j \leq n$, and $i=j-n$, $k=2$ for $j > n$. Vice versa, we compute $k_j := \lceil \ell_{i,k}/n \rceil$ and $p_j := \ell_{i,k} - n r_j$, where $j=i$ for $k=1$ and $j=i+n$ for $k=2$. These maps are inverse to each other. It remains to show that the image of a cell is an affine Dellac configuration to the parameter $\omega$, and that each of these configurations is mapped to a cell.

Given a length tuple parametrising the fixed point $U(\mb{l})$, which describes a cell in the quiver Grassmannian, we compute the numbers $k_j$ for $j \in [2n]$ as described above. By construction of the map, we obtain exactly $2n$ parameters $k_j \in [\omega]_0$ and have a unique way to write them in the $2n$ rows of a configuration. These parameters satisfy 
\[ \sum_{j=1}^{2n}(p_j+nr_j) = \sum_{i=1}^{n} \ell_{i,1} + \ell_{i,2} = \omega n^2, \]
since the entries in the dimension vector of $U(\mb{l})$ are equal to $\omega n$. It follows from the parametrisation of cells by successor closed subquivers, that one cell is obtained from each other cell by moving parts of subsegments in the coefficient quiver. These movements preserve the number of segments starting and ending over each vertex of the underlying quiver, where the empty segment over the vertex $i$ is considered as segment starting over the vertex $i+1$. By Lemma~\ref{lma:irr_comp} the top-dimensional cells have two segments starting and ending over each vertex of the quiver $\Delta_n$. Hence this is true for all other cells such that in the configuration, which is assigned to the length tuple $\mb{l}$, there are exactly two entries in each column. Accordingly, the image of a cell is an affine Dellac configuration to the parameter $\omega$.

Starting with an affine Dellac configuration $D(\mb{k})$ to the parameter $\omega$, we compute the numbers $\ell_{i,k}$ as described above. We obtain $2n$ numbers parametrising the length of the subsegments in the coefficient quiver of $M_{\omega}^a$, and by construction all of these numbers live in the set $[\omega n]_0$. It remains to show that the subquiver, which is parametrised by this length tuple, has dimension vector $\mb{e}_{\omega} =  (e_i=\omega n)_{i \in \Z_n}$. 

From an affine Dellac configuration, we can compute the dimension vector of the corresponding successor closed subquiver as follows. In the $j$-th row we fill the boxes on the right of the box with the entry $k_j$ and on the left of the separator with the number $k_j$, and in all remaining boxes of this row we write $r_j$. The row vector as obtained by this procedure equals the dimension vector of the quiver representation $U(i,\ell_{i,k})$, where $i$ and $k$ are obtained as above. Summing all row vectors computed from the configuration, we obtain the dimension vector of the associated quiver representation.

Property $(1)$ and $(2)$ of affine Dellac configurations imply that over each vertex of $\Delta_n$ there are starting and ending exactly two segments of the associated subquiver. Hence the entries of the dimension vector of the subquiver are all the same and equal to $\omega n$, since their sum is equal to $\omega n^2$ by Property $(3)$ of affine Dellac configurations. Thus the image of an affine Dellac configuration is a torus fixed point, which parametrises a cell in the quiver Grassmannian.
%-------------------------------------------------------------------------------------
\section{Linear Degenerations of the Affine Flag Variety}\label{sec:lin-deg-flag}
%-------------------------------------------------------------------------------------
In this section, we study linear degenerations of the affine flag variety, following the approach of G.~Cerulli Irelli, X.~Fang, E.~Feigin, G.~Fourier and M.~Reineke as introduced in \cite{CFFFR2017}. %This generalises Definition~\ref{def:deg-aff-flag} of the degenerate affine flag variety. 
In contrast to toric degenerations, where methods from toric geometry are used to study the degenerations and deduce properties of the original variety, linear degenerations are tailored to be studied as quiver Grassmannians. For more details we refer to the overview \cite{FaFoLi2016}.
%-------------------------------------------------------------------------------------
\subsection{Rotation Invariant Parametrisation of the Affine Flag Variety}
%-------------------------------------------------------------------------------------
In Proposition~\ref{prop:alt-param-classical} we obtained an alternative parametrisation of the affine flag variety, which is based on Sato Grassmannians. Using the fact that $s_k \mathrm{SGr}_0 \cong \mathrm{SGr}_k$, we get a rotation invariant parametrisation of the affine flag variety. 
\begin{prop}
\label{prop:alternative_parametrisation}
The affine flag variety $\mathcal{F}l\big(\widehat{\mathfrak{gl}}_n\big)$, as subset in the product of $n$ copies of the Sato Grassmannian $\mathrm{SGr}_0$, is parametrised as
\[\mathcal{F}l\big(\widehat{\mathfrak{gl}}_n\big) \cong \Big\{ \big(U_i \big)_{i\in \Z_n}  \in \mathrm{SGr}_0^{\times n}  \ : \  s_{-1} U_i \subseteq  U_{i+1} \fa i \in \Z_n \Big\}. \]
\end{prop}
If we replace the shifts $s_{-1}$ by arbitrary linear maps, we obtain linear degenerations of the affine flag variety.
\begin{defi}\label{def:lin-deg-aff-flag}
For a tuple $f=(f_i)$ of linear maps $f_i : V \to V$ with $i \in \Z_n$, the \textbf{f-linear degenerate affine flag variety} is defined as
\[\mathcal{F}l^f\big(\widehat{\mathfrak{gl}}_n\big) := \Big\{ \big(U_i \big)_{i\in \Z_n}  \in \mathrm{SGr}_0^{\times n}  \ : \  f_i U_i \subseteq  U_{i+1} \fa i \in \Z_n \Big\}. \]
\end{defi}
The degeneration as in Definition~\ref{def:deg-aff-flag} corresponds to $f_i = s_{-1} \circ \mr{pr}_{1} $. In the rest of this section we study the geometry of the linear degenerations via quiver Grassmannians for the equioriented cycle and parametrise the isomorphism classes of linear degenerations. The rotation invariance of the affine flag as in Proposition~\ref{prop:alternative_parametrisation} simplifies the study of isomorphism classes for the linear degenerations.
%-------------------------------------------------------------------------------------
\subsection{Finite Approximations of the Linear Degenerations}\label{sec:fin-approx-lin-deg}
%-------------------------------------------------------------------------------------
%In the same way as for $\mathcal{F}l^{a}\big(\widehat{\mathfrak{gl}}_n\big)$, we define finite approximations of the linear degenerations, and identify them with quiver Grassmannians. 
For a positive integer $\omega$, the finite approximation of the $f$-linear degenerate affine flag variety is defined as
\[ \mathcal{F}l^f_\omega\big(\widehat{\mathfrak{gl}}_n\big)  := \Big\{  (U_i)_{i \in \Z_n} \in \mathcal{F}l^f\big(\widehat{\mathfrak{gl}}_n\big) \ : \ V_{-\omega n} \subseteq U_i \subseteq V_{\omega n} \fa i \in \Z_n \Big\} .\] 
In order to view this finite approximation as quiver Grassmannian, define the finite dimensional vector space 
\[ V^{(\ell)} := \mr{span}\big( v_\ell, v_{\ell-1}, \dots , v_{-\ell+2}, v_{-\ell+1} \big) \quad \mr{for} \ \ell \in \N. \]
\begin{trm}
\label{trm:finite_approx-lin-deg}
Let $\omega \in \N$ be given, define the quiver representation
\[ M_\omega := \Big( \, \big( V_i:= V^{(\omega n)} \big)_{i \in \Z_n}, \big( M_{\alpha_i} := f_i\vert_{V^{(\omega n)}} \big)_{i \in \Z_n} \,  \Big),\] and the dimension vector $\mb{e}_\omega := (e_i := \omega n)_{i \in \Z_n}$. Then the finite dimensional approximation of the $f$-linear degenerate affine flag variety is isomorphic to the quiver Grassmannian corresponding to $M_\omega$ and $\mb{e}_\omega$, i.e.
\[ \mathcal{F}l^f_\omega\big(\widehat{\mathfrak{gl}}_n\big) \cong \mr{Gr}^{\Delta_n}_{\mb{e}_\omega}\big(M_\omega\big). \]
\end{trm}
\begin{proof}
Analogous to the proof of Theorem~\ref{trm:finite_approx-in-the-intro} we obtain
\begin{align*}
 \mathcal{F}l^f_\omega\big(\widehat{\mathfrak{gl}}_n\big)  &= \Bigg\{ \big(U_i \big)_{i\in \Z_n}  \in \mathrm{SGr}_{0,\omega}^{\times n}  \ : \  f_i U_i \subseteq  U_{i+1} \fa i \in \Z_n \Bigg\}\\
 &\cong \Bigg\{ \big(U_i \big)_{i \in \Z_n}  \in \prod_{i \in \Z_n} \mathrm{Gr}_{\omega n}\big(V^{(\omega n)}\big) \  : \  f_i\vert_{V^{(\omega n)}} U_i \subseteq  U_{i+1} \fa i \in \Z_n \Bigg\}\\
 &\cong \mr{Gr}^{\Delta_n}_{\mb{e}_\omega}\big(M_\omega\big).
\end{align*} 
\end{proof}
\begin{rem}
This theorem implies that all approximations of linear degenerations of the affine flag variety can be studied using quiver Grassmannians for the equioriented cycle. However it is only possible to apply the methods introduced in the present paper if the quiver representation $M_\omega$ as in Theorem~\ref{trm:finite_approx-lin-deg} is nilpotent.
\end{rem}
%-------------------------------------------------------------------------------------
\subsection{Nilpotent Linear Degenerations}
%-------------------------------------------------------------------------------------
In this subsection we characterise the degenerations which can be studied with the methods introduced in this article. An endomorphism $f \in \mr{End}^{\times n}(V)$ is called \f{nilpotent} if for all $\ell \in \N$ there exists a $k \in \N$ such that the concatenation of the restrictions of $f_i$ to $V^{(\ell)}$, along the arrows of the cycle, vanish for each stating vertex $i \in \Z_n$ and paths of length $k$, i.e.:
\[ \big(f_{i+k}\circ \dots \circ f_{i+1}\circ f_{i}\big)\big\vert_{V^{(\ell)}} = 0.\]
The set of nilpotent endomorphism is denoted by $\mr{End}^{\times n}_{nil}(V)$.
\begin{rem}\label{rem:nilpot-quiver-rep}
If $f \in \mr{End}^{\times n}(V)$ is nilpotent, the quiver representation $M_\omega$ as in Theorem~\ref{trm:finite_approx-lin-deg} is nilpotent for every finite approximation and we can apply the results from Section~\ref{sec:quiver-grass-cycle} to the corresponding quiver Grassmannians. 
\end{rem}
In particular, we obtain a cellular decomposition of each approximation.
\begin{trm}\label{trm:cell-decomp-lin-deg}
Let $f \in \mr{End}^{\times n}_{nil}(V)$ and $\omega \in \N$, then $\mathcal{F}l^f_\omega\big(\widehat{\mathfrak{gl}}_n\big)$ admits a cellular decomposition into the attracting sets of its $\C^*$-fixed points.
\end{trm}
\begin{proof}
By Remark~\ref{rem:nilpot-quiver-rep} the quiver representation $M_\omega$ is nilpotent if $f \in \mr{End}^{\times n}_{nil}(V)$. Hence we can apply Theorem~\ref{trm:cell_decomp-approx-lin-deg-aff-flag} to the corresponding approximation of the $f$-linear degenerate flag variety, by Theorem~\ref{trm:finite_approx-lin-deg}. 
\end{proof}
\begin{rem}
Analogous to \cite[Remark~6]{CFFFR2017}, Theorem~\ref{trm:finite_approx-lin-deg} allows to compute the Euler characteristic and the Poincar\'e polynomial of each finite approximation, corresponding to an $f \in \mr{End}^{\times n}_{nil}(V)$.
\end{rem}
\begin{rem}\label{rem:nilpotent-locus}
For the linear degenerations, which do not correspond to nilpotent endomorphisms, it is not possible to use the methods in the present paper to deduce results about their geometry. Hence it would be of interest to understand quiver Grassmannians for the cycle, with not necessarily nilpotent representations. %Hence it would be of interest to understand the geometry of quiver Grassmannians for the cycle, with not necessarily nilpotent representations.
\end{rem}
%-------------------------------------------------------------------------------------
\subsection{Isomorphism Classes of Linear Degenerations}
%-------------------------------------------------------------------------------------
On tuples of linear maps $f \in \mr{End}^{\times n}(V)$, we have an action of $g = (g_i)_{i \in \Z_n}$ in $G := \prod_{i \in \Z_n} \mr{GL}(V)$ via base change, i.e.
\[ g.f := \Big(  g_{i+1}f_{i}g_{i}^{-1} \Big)_{i \in \Z_n}.\]
Two linear degenerations are isomorphic if and only if their defining tuples of maps live in the same $G$-orbit. This follows directly from the parametrisation of the degenerations via Sato Grassmannians as in Definition~\ref{def:lin-deg-aff-flag}. Hence it is sufficient to study one representative for each $G$-orbit.  

For the restrictions of the map tuples $f=(f_i)_{i \in \Z_n}$ to $V^{(\ell)}$, these $G$-orbits where studied by G.~Kempken in her thesis \cite{Kempken1982}. In each finite approximation $\mr{End}^{\times n}(V^{(\ell)})$, the number of $G$-orbits is finite, whereas there are infinitely many different $G$-orbits in $\mr{End}^{\times n}(V)$. This makes it more complicated to characterise the isomorphism classes of degenerations in general. 

The $G$-action does not effect the corank of the maps in the tuples and the corank of their concatenations, since the matrices $g_i$ are invertible and by definition their action cancels out for the middle terms of concatenations along paths in the quiver. Hence, each $G$-orbit is characterised by these coranks for one representing map tuple. By Remark~\ref{rem:nilpotent-locus}, the corresponding degenerations can be studied with the methods from the present paper if and only if the map tuple is nilpotent. In the next subsection we give one special type of degenerations, where our methods apply and which is independent of the approximation parameter $\omega$.
%-------------------------------------------------------------------------------------
\subsection{Partial Degenerations and their Corank Tuples}
%-------------------------------------------------------------------------------------
In this subsection, we restrict us to intermediate degenerations between the non-degenerate affine flag variety and its degeneration as in Definition~\ref{def:deg-aff-flag}. These degenerations correspond to map tuples in the $G$-orbits of $f \in \mr{End}^{\times n}_{nil}(V)$, where each $f_i$ is either the shifted projection $s_{-1} \circ\mr{pr}_{1}$ or the index shift $s_{-1}$. 

Since degenerations corresponding to map tuples in one $G$-orbit are isomorphic, the intermediate degenerations are completely determined by the corank of the maps $f_i$, which is independent of the approximation parameter $\omega$. Hence it is sufficient to view corank tuples $\mb{c} \in \{1,2\}^{\Z_n}$. The corank tuple $\mb{c}^{(2)} = (2,\dots,2)$ corresponds to the degeneration from Definition~\ref{def:deg-aff-flag}, and the tuple $\mb{c}^{(1)} =(1,\dots,1)$ parametrises the non-degenerate affine flag variety as in Proposition~\ref{prop:alternative_parametrisation}. 

Let $\mathcal{F}l^{\mb{c}}\big(\widehat{\mathfrak{gl}}_n\big)$ be a representative, for the isomorphism class of linear degenerations of the affine flag variety, corresponding to the corank tuple $\mb{c} \in \{1,2\}^{\Z_n}$. In Section~\ref{sec:part-deg-dellac}, it is shown that these degenerations play a special role among the linear degenerations of the affine flag variety. Namely, they are the maximal subclass of degenerations, where it is possible to adapt the parametrisation of cells via affine Dellac configurations.
\begin{trm}
\label{lma:approx-lin-deg-aff-flag}
For $\omega \in \N$ and $\mb{c} \in \{1,2\}^{\Z_n}$ the finite approximation is given as
\[ \mathcal{F}l^{\mb{c}}_{\omega}\big(\widehat{\mathfrak{gl}}_n\big) \cong \mr{Gr}^{\Delta_n}_{\mb{e}_{\omega}}\big(M_{\omega}^\mb{c}\big), \]
where $\mb{e}_{\omega} := \bdim \bigoplus_{i \in \mathbb{Z}_n} U_i(\omega n) = (\omega n)_{i \in \Z_n}$ and for every $i \in \Z_n$ the representation $M_{\omega}^\mb{c}$ contains the summand $U_i(\omega n) \otimes \C^2$ if ${c}_{i}=2$ or $U_{i}(2 \omega n)$ if ${c}_{i}=1$. 
\end{trm}
\begin{prop}\label{prop:re-param-M_w}
The quiver representation $M_{\omega}^\mb{c}$ is isomorphic to the quiver representation
\[ \Big( \, \big( V_i:= V^{(\omega n)} \big)_{i \in \Z_n}, \big( \, M_{\alpha_i} := s_{-1} \circ \mr{pr}_{1}^{c_{i}-1} \, \big)_{i \in \Z_n} \,  \Big). \]
\end{prop}
\begin{proof}
For the representation $M_\omega^\mb{c}$, the vector space over each vertex $i \in \Z_n$ has dimension $2\omega n$, which is also the dimension of $V_i$. In the coefficient quiver of $M_{\omega}^\mb{c}$, there are $c_i$ segments starting over the vertex $i \in \Z_n$. 

The first segment is starting in the fist point over the vertex $i$, and in the $k$-th step its arrow goes from the $k$-th point over the vertex $i+k-1$ to the $k+1$-th point over the vertex $i+k$. If $c_i=2$, this segment has length $\omega n$ and there has to be a second segment starting over the same vertex. If $c_i=1$, this segment has length $2\omega n$ and there is no second segment starting over the vertex $i \in \Z_n$. 

Now assume that $c_i=2$. The first segment ends in the $\omega n$-th point over the vertex $i-1$, and it is not possible that there exists an arrow pointing to the $\omega n+1$-th point over the vertex $i$. We choose this point as starting point for the second segment starting over the vertex $i$.  

In the $k$-th step the arrow of this segment goes from the $\omega n +k$-th point over the vertex $i+k-1$ to the $\omega n +k+1$-th point over the vertex $i+k$, and it ends in the $2\omega n$-th point over the vertex $i+n-1=i-1$. With this realisation of the coefficient quiver of $M_{\omega}^\mb{c}$, we have the map \(s_{1} \circ \mr{pr}_{\omega n}^{c_{i}-1},\)
for the arrow $\alpha_i$ from vertex $i$ to vertex $i+1$. The claim follows with the base change as described in the proof of Theorem~\ref{trm:finite_approx-in-the-intro} as given in Section~\ref{sec:iso:fin-approx-quiv-grass}.  
\end{proof}
This yields an alternative parametrisation for the quiver Grassmannian, which can be identified with the approximation of the partial degeneration.
\begin{proof}[Proof of Theorem~\ref{lma:approx-lin-deg-aff-flag}]
Theorem~\ref{trm:finite_approx-lin-deg} applies to the approximations of the partial degenerations and the desired parametrisation of the corresponding quiver representation is given in Proposition~\ref{prop:re-param-M_w}.  
\end{proof}
%-------------------------------------------------------------------------------------
\subsection{Partial Degenerations of Affine Dellac Configurations}\label{sec:part-deg-dellac}
%-------------------------------------------------------------------------------------
In this section, we introduce subsets of affine Dellac configurations, which describe the cells in the approximations of the partial degenerate affine flag varieties. For every $i \in \Z_n$, the representation $M_{\omega}^\mb{c}$ contains the summand 
\[ U(i;\omega n) \otimes \C^2 \ \  \mr{if} \ {c}_{i}=2 \quad \quad  \mr{or} \quad \quad U(i;2\omega n) \ \ \mr{if} \ {c}_{i}=1.\]
Recall that for $\omega \in \N$ the finite approximation is given as
\[ \mathcal{F}l^{\mb{c}}_{\omega}\big(\widehat{\mathfrak{gl}}_n\big) \cong \mr{Gr}^{\Delta_n}_{\mb{e}_{\omega}}\big(M_{\omega}^\mb{c}\big), \]
where $\mb{e}_{\omega} := (e_i = \omega n)_{i \in \Z_n}$.  

Since the quiver representation $M_{\omega}^\mb{c}$ is nilpotent, we can apply Theorem~\ref{trm:cell_decomp-approx-lin-deg-aff-flag} to obtain a cellular decomposition of $\mathcal{F}l^{\mb{c}}_{\omega}\big(\widehat{\mathfrak{gl}}_n\big)$. %, which is induced by a $\C^*$-action. 
By \cite[Proposition~1]{Cerulli2011}, the cells in the approximations are in bijection with successor closed subquivers in the coefficient quiver of $M_\omega^{\mb{c}}$, with $\omega n$ marked points over each vertex. 

Accordingly, these successor closed subquivers are parametrised by the set
\[ \mathcal{C}_{n}^{\mb{c}}(\omega):= \Big\{ \mb{l} := (\ell_{i,k}) \in \bigoplus_{i \in \Z_n} \bigoplus_{k=1}^{c_i} \big[(3-c_i)\omega n \big]_0 : \ \bdim U(\mb{l}) = \mb{e}_{\omega} \Big\}, \]
where 
\[ U(\mb{l}) := \bigoplus_{i \in \Z_n}  \bigoplus_{k=1}^{c_i} U(i;\ell_{i,k}). \]
\begin{defi}
An affine Dellac configuration $\widehat{D} \in \widehat{DC}_n(\omega)$ is $\mb{c}$\f{-degenerate} to the parameter $\mb{c} \in \{1,2\}^{\Z_n}$ if $k_j > 0$ for $j \in [n]$ implies that $k_{j+n}=\omega$ and $p_{j+n} =n$, whenever $c_j=1$. The set of all $\mb{c}$-degenerate affine Dellac configurations is denoted by $\widehat{DC}\,^\mb{c}_n(\omega)$.
\end{defi}
These configurations parametrise the cells in the partial degenerations.
\begin{trm}
\label{trm:bij-cells-lin-deg-aff-dellac-conf}
For $\omega \in \N$ and $\mb{c} \in \{1,2\}^{\Z_n}$, the cells in the approximation $\mathcal{F}l^{\mb{c}}_{\omega}\big(\widehat{\mathfrak{gl}}_n\big)$ of the partial degenerate affine flag variety are in bijection with $\mb{c}$-degenerate affine Dellac configurations to the parameter $\omega$.
\end{trm}
The proof of this statement is analogous to the proof of Theorem~\ref{trm:bij-cells-aff-delllac} in the setting with the degeneration as in Definition~\ref{def:deg-aff-flag}. Here it is crucial that the number of indecomposable summands of $M_{\omega}^\mb{c}$ is less or equal to $2n$ and that there are at most two indecomposable summands ending over each vertex of $\Delta_n$. This allows to parametrise the summands by the rows of affine Dellac configurations. This fails for degenerations, where the corank of some $f_i$ is strictly bigger than two. In this case there are at least three summands ending over some vertex of $\Delta_n$ and the structure of the corresponding quiver representations does not match the structure of affine Dellac configurations any more. 
%-------------------------------------------------------------------------------------
\subsection{Linear Degenerations of higher Corank}\label{sec:lin-deg-higher-corank}
%-------------------------------------------------------------------------------------
It turns out that the degeneration from Definition~\ref{def:deg-aff-flag} and the non-degenerate affine flag variety are very special, even among the degenerations corresponding to nilpotent endomorphism. In general, it is only possible to apply Theorem~\ref{trm:framed-module}, to the quiver Grassmannians from Theorem~\ref{trm:finite_approx-lin-deg}, in these two cases. This holds since already for the partial degenerations, the indecomposable summands of $M_\omega^{\mb{c}}$ are of different length, apart from the extremal cases $\mb{c}^{(1)}$ and $\mb{c}^{(2)}$. As pointed out in the end of the previous section, this does not get better for linear degenerations of higher corank. Hence we can not apply Lemma~\ref{lma:rat_sing}, to the linear degenerations. Accordingly, with the methods from the present paper, it is not possible to derive any statement about the geometry of the linear degenerations, apart from the cellular decomposition in the nilpotent locus as in Theorem~\ref{trm:cell-decomp-lin-deg}, and the parametrisation of the cells by Dellac configurations for the partial degenerations in Theorem~\ref{trm:bij-cells-lin-deg-aff-dellac-conf}. Thus it would be of great interest to have a better understanding for the geometry of arbitrary quiver Grassmannians for the cycle.
%-------------------------------------------------------------------------------------
\subsection{Poincar\'e Polynomials of the Approximations of Partial Degenerations}\label{sec:p-poly-lin-deg-aff-flag}
%------------------------------------------------------------------------------------
There exists a function 
\begin{align*} 
h^\mb{c}: \widehat{DC}\,^\mb{c}_n(\omega) &\longrightarrow \ \Z\\
D \ \ &\longmapsto \ h^\mb{c}(D)
\end{align*}
such that $h^\mb{c}(D)$ is equal to the complex dimension of the corresponding cell, in the approximation of the partial degenerate affine flag variety \cite[Theorem 6.61]{Pue2019}. 
\begin{trm}
\label{trm:p-poly-lin-deg-aff-flag}
For $\omega \in \N$ and $\mb{c} \in \{0,1\}^n$, the Poincar\'e polynomial of $\mathcal{F}l^{\mb{c}}_{\omega}\big(\widehat{\mathfrak{gl}}_n\big)$ is given by
\[p_{ \mathcal{F}l^{\mb{c}}_{\omega}\big(\widehat{\mathfrak{gl}}_n\big)  }(q) = \sum_{ D \in \widehat{DC}\,^\mb{c}_n(\omega) } q^{h^\mb{c}(D)}.\]
\end{trm}
The precise formula for the function $h^\mb{c}$ is very complicated and only practical for computer programs. Hence we decide to omit this detail and highlight the description of the Poincar\'e polynomials, which is easier to handle in small examples. For the alternative approach, we draw the successor closed subquivers as described in Section~\ref{sec:cell-decomp}, and count the holes below the starting points of the segments as in \cite[Remark~6]{CFFFR2017}. By Theorem~\ref{trm:cell_decomp-approx-lin-deg-aff-flag}, this even works for all quiver Grassmannians associated to nilpotent representations of the cycle. Based on the second approach, we implemented a program to compute the Poincar\'e polynomials of some approximations using SageMath \cite{Sage}. The explicit dimension function $h^\mb{c}$ and the code for the computer program can be found in \cite[Chapter~6.10, Appendix~B]{Pue2019}. The formula based on the dimension function for the affine Dellac configurations might lead to a more efficient computer program, but the algorithmic version based on the coefficient quivers is efficient enough to compute examples for $n \leq 5$ and $\omega \leq 6$.
%-------------------------------------------------------------------------------------
\section*{Acknowledgements} 
%------------------------------------------------------------------------------------
This research was funded by the DFG/RSF project "Geometry and representation theory at the interface of Lie algebras and quivers". Furthermore, I acknowledge the PRIN2017 CUP E8419000480006, and the MIUR Excellence Department Project awarded to the Department of Mathematics, University of Rome Tor Vergata, CUP E83C18000100006. I want to thank M.~Reineke for many very inspiring discussions during the preparation of this work and M.~Lanini for very helpful suggestions to correct a mistake in a previous version. Moreover I am grateful to an anonymous referee for very useful suggestions to optimise the structure of this article.
%-------------------------------------------------------------------------------------
%\section[Quellen]{Referenzen}
%-------------------------------------------------------------------------------------

\end{document}